\documentclass[article]{siamart190516}
\usepackage{url}
\usepackage{amssymb}
\usepackage{amsmath}
\usepackage{stmaryrd}
\usepackage{siunitx}
\usepackage{commath}
\usepackage{graphicx}
\usepackage[caption=false]{subfig}

\usepackage{lipsum}
\usepackage{amsfonts}
\usepackage{epstopdf}
\usepackage{algorithm,algorithmic}
\usepackage{bm}
\usepackage{amsopn}

\ifpdf
  \DeclareGraphicsExtensions{.eps,.pdf,.png,.jpg}
\else
  \DeclareGraphicsExtensions{.eps}
\fi
\newsiamremark{remark}{Remark}
\newsiamremark{hypothesis}{Hypothesis}
\crefname{hypothesis}{Hypothesis}{Hypotheses}
\newsiamthm{claim}{Claim}
\title{Local Fourier analysis of multigrid for hybridized and embedded
  discontinuous Galerkin methods\thanks{\funding{SR and HDS gratefully
      acknowledge support from the Natural Sciences and Engineering
      Research Council of Canada through the Discovery Grant program
      (RGPIN-05606-2015 and RGPIN-2019-04155).}}}
\author{Yunhui He\thanks{Department of Applied Mathematics University
    of Waterloo, 200 University Ave W,Waterloo, ON N2L 3G1, Canada
    (\email{yunhui.he@uwaterloo.ca}, \email{srheberg@uwaterloo.ca},
    \email{hdesterck@uwaterloo.ca}).}  \and Sander
  Rhebergen\footnotemark[2] \and Hans De Sterck\footnotemark[2]}
\headers{LFA for HDG and EDG}{Y. He, S. Rhebergen, and H. De Sterck}
\begin{document}
%------------------------------------------------------------------------------
\maketitle
%------------------------------------------------------------------------------
\begin{abstract}
  In this paper we present a geometric multigrid method with
  Jacobi and Vanka relaxation for hybridized and embedded discontinuous
  Galerkin discretizations of the Laplacian. We present a local
  Fourier analysis (LFA) of the two-grid error-propagation operator
  and show that the multigrid method applied to an embedded
  discontinuous Galerkin (EDG) discretization is almost as efficient
  as when applied to a continuous Galerkin discretization. We
  furthermore show that multigrid applied to an EDG discretization
  outperforms multigrid applied to a hybridized discontinuous Galerkin
  (HDG) discretization. Numerical examples verify our LFA predictions.
\end{abstract}
%------------------------------------------------------------------------------
%------------------------------------------------------------------------------
%------------------------------------------------------------------------------
\section{Introduction}
\label{sec:intro}

The hybridizable discontinuous Galerkin (HDG) method was introduced in
\cite{Cockburn:2009} with the purpose of reducing the computational cost
of discontinuous Galerkin (DG) methods while retaining the
conservation and stability properties of DG methods. This is achieved
by introducing facet variables and eliminating local (element-wise)
degrees-of-freedom. This static condensation can significantly reduce
the size of the global problem. Indeed, it was shown in
\cite{Kirby:2012,Yakovlev:2016} that the HDG method either outperforms
or demonstrates comparable performance when compared to the CG
method. This is in part also due to the local postprocessing which
allows one to obtain a superconverged solution. However, they mention
that these results hold true \emph{only} when a direct solver is used;
when an iterative solver is used the HDG method falls behind
performance-wise.

The literature on iterative solvers and preconditioners for CG
discretizations is vast. In contrast, there are only few studies on
solvers for HDG and hybridized discretizations. We mention, for
example, \cite{Gopalakrishnan:2009} which presents a convergence
analysis of multigrid for a hybridized Raviart--Thomas discretization,
\cite{Cockburn:2014} which analyzes an auxiliary space multigrid
method for HDG discretizations of elliptic partial differential
equations, \cite{Fabien:2019} which considers parallel geometric
multigrid for HDG methods, \cite{Wildey:2019} which presents a unified
geometric multigrid method for hybridized finite element methods, and
\cite{Dobrev:2019} which considers the solution of hybridized systems
by algebraic multigrid. Furthermore, a performance comparison of a
variation of the multigrid method proposed in \cite{Cockburn:2014}
applied to CG, HDG and DG discretizations for the Poisson problem is
conducted in \cite{Kronbichler:2018}. They show that high-order
continuous finite elements give the best time to solution for smooth
solutions followed by matrix-free solvers for DG and HDG (using
algebraic multigrid).

An alternative to HDG methods is the embedded discontinuous Galerkin
(EDG) method which was introduced and analyzed in
\cite{Cockburn:2009b,Guzey:2007,Wells:2011}. The difference between an
EDG and HDG method is that the facet variables in an EDG method are
continuous between facets; in the HDG method they are discontinuous
between facets. For the EDG method this means that after static
condensation it has the same number of global degrees-of-freedom
(DOFs) as a continuous Galerkin (CG) method and less DOFs than an HDG
method on a given mesh. The EDG method, however, does not have the
superconvergent properties of the postprocessed HDG solution and has
therefore been studied less in the literature. However, we will see
that the algebraic structure of the linear system resulting from an
EDG method is better suited to fast iterative solvers than the linear
system resulting from an HDG discretization. Indeed, this paper is
motivated by the observation that multigrid methods applied to EDG
discretizations of the Laplacian outperform multigrid methods applied
to HDG discretizations of the Laplacian. This was observed, for
example, in the context of the Stokes problem in \cite{Rhebergen:2020}
using the block-preconditioners developed in \cite{Rhebergen:2018b}.

The first goal of this paper is to present a geometric multigrid
method for the hybridized and embedded discontinuous Galerkin
discretizations of the Laplacian. The challenge in designing efficient
multigrid methods for these discretizations is twofold. Firstly,
since the facet function spaces do not form a nested hierarchy on
refined grids, the design of intergrid transfer operators is not
trivial. We therefore use the recently introduced Dirichlet-to-Neumann
(DtN) maps proposed in \cite{Wildey:2019} for hybridized finite
element methods. The second challenge is the design of an efficient
relaxation scheme. The smoothers used in multigrid methods applied to
discontinuous Galerkin (DG) discretizations of an elliptic PDE are
usually the classical (block) Jacobi and (block) Gauss--Seidel
smoother (see for example \cite{Gopalakrishnan:2003, Hemker:2003,
  Hemker:2004, Raalte:2005}). We, however, use additive Vanka-type
relaxation \cite{farrell2019local,Frommer:1999Vanka,Saad2003book}
since Vanka-type relaxation is more suitable for parallel
computing on general meshes than (block) Gauss--Seidel and, as we will
show, results in a multigrid method that requires less iterations than when
using Jacobi relaxation. Using Vanka-type relaxation requires the definition
of Vanka-patches. We will study two different types of patches,
namely, element-wise and vertex-wise patches. We remark that
Vanka-type relaxation has been studied also in the context of
discontinuous Galerkin methods for the Stokes problem in
\cite{Adler:2017}.

The second and main goal of this paper is to present a two-dimensional
local Fourier analysis (LFA) of the geometric two-grid error
propagation operator of HDG and EDG discretizations of the
Laplacian. Local Fourier analysis is used to predict the efficiency of
multigrid methods. We will show that the performance of multigrid
applied to an embedded discontinuous Galerkin discretization is
similarly efficient and scalable to when multigrid is applied to a
continuous Galerkin discretization. We furthermore show that multigrid
applied to EDG outperforms multigrid applied to HDG, confirming what
was previously observed but not explained in \cite{Rhebergen:2020}.
We remark that local Fourier analysis has previously been used to
study the performance of multigrid applied to discontinuous Galerkin
discretizations (see for example \cite{Fidkowski:2005, Hemker:2003,
  Hemker:2004, Vegt:2012a, Vegt:2012b, Raalte:2005}). One-dimensional
LFA has been used also to study a multilevel method for the HDG
discretization of the Helmholtz equation \cite{Chen:2014}. However, to
the best of our knowledge, a two-dimensional LFA has not been applied
in the context of hybridizable and embedded discontinuous Galerkin
discretizations, and the performance of geometric multigrid for these
types of methods has not been analyzed before.

The remainder of this work is organized as follows. In
\cref{sec:discretization} we discuss the hybridized and embedded
discontinuous Galerkin discretization of the Poisson problem. We
present geometric multigrid with additive Vanka relaxation for the
hybridized and embedded trace system in \cref{sec:two-grid}. A local
Fourier analysis of the corresponding two-grid method is presented in
\cref{sec:LFA-tg}. Our theory is applied and verified by numerical
examples in \cref{sec:numer} and we draw conclusions in
\cref{sec:conclu}.

%------------------------------------------------------------------------------
\section{The EDG and HDG methods}
\label{sec:discretization}

In this section we present EDG and HDG methods for the Poisson
problem:
\begin{subequations}
  \begin{align}
    \label{eq:laplace}
    -\Delta u &= f && \text{in } \Omega,
    \\
    u &= 0 && \text{on } \partial\Omega,
  \end{align}
  \label{eq:poisson}
\end{subequations}
where $\Omega \subset \mathbb{R}^2$ is a bounded polygonal domain with
boundary $\partial\Omega$, $f : \Omega \to \mathbb{R}$ is a given
source term, and $u:\Omega \to \mathbb{R}$ is the unknown.

%------------------------------------------------------------------------------
\subsection{The discretization}

We will discretize \cref{eq:poisson} by an EDG and an HDG method. For
this, denote by $\mathcal{T}_h := \cbr[0]{K}$ a tesselation of
$\Omega$ into non-overlapping quadrilateral elements $K$. We will denote the
diameter of an element $K$ by $h_K$ and the maximum diameter over all
elements $K \in \mathcal{T}_h$ by $h$. The boundary of an element is
denoted by $\partial K$ and the outward unit normal vector on
$\partial K$ is denoted by $\boldsymbol{n}$. An interior face
$F := \partial K^+ \cap \partial K^-$ is shared by two adjacent
elements $K^+$ and $K^-$ while a boundary face is a part of $\partial K$
that lies on $\partial\Omega$. We denote the set of all faces by
$\mathcal{F}_h := \cbr[0]{F}$ and the union of all faces by
$\Gamma^0_h$.

Denote by $Q^k(K)$ and $Q^k(F)$ the spaces of tensor product
polynomials of degree $k$ on, respectively, element $K$ and face
$F$. We consider the following discontinuous finite element function
spaces:
\begin{align}
  V_h       &:= \cbr[0]{v_h \in L^2(\Omega)\, :\, v_h \in Q^k(K)\ \forall K \in \mathcal{T}_h},
  \\
  \bar{V}_h &:= \cbr[0]{\bar{v}_h \in L^2(\Gamma^0_h)\, :\, \bar{v}_h \in Q^k(F)\
              \forall F \in \mathcal{F}_h,\ \bar{v}_h = 0 \text{ on } \partial\Omega}.
\end{align}
For the EDG and HDG methods we then define
\begin{equation}
  \label{eq:edghdgspace}
  \boldsymbol{X}_h := V_h \times \bar{X}_h \quad \text{with} \quad
  \bar{X}_h :=
  \begin{cases}
    \bar{V}_h & \text{HDG method},
    \\
    \bar{V}_h \cap C^0(\Gamma^0_h) & \text{EDG method}.
  \end{cases}
\end{equation}
We note that the HDG method uses discontinuous facet function spaces
and that the EDG method uses continuous facet function spaces.

For notational purposes we denote function pairs in $\boldsymbol{X}_h$
by $\boldsymbol{v}_h := (v_h, \bar{v}_h) \in \boldsymbol{X}_h$. For
functions $u, v \in L^2(K)$ we write $(u, v)_K := \int_K uv \dif x$
and define
$(u, v)_{\mathcal{T}_h} := \sum_{K \in \mathcal{T}_h}(u,
v)_K$. Similarly, for functions $u, v \in L^2(E)$ where
$E \subset \mathbb{R}$ we write
$\langle u, v \rangle_E := \int_E uv \dif x$ and
$\langle u, v \rangle_{\partial\mathcal{T}_h} := \sum_{K \in
  \mathcal{T}_h}\langle u, v \rangle_{\partial K}$.

The interior penalty EDG and HDG methods are given by
\cite{Cockburn:2009,Cockburn:2009b,Wells:2011}: find
$\boldsymbol{u}_h \in \boldsymbol{X}_h$ such that
\begin{equation}
  \label{eq:edghdg}
  a_h(\boldsymbol{u}_h, \boldsymbol{v}_h) = (v_h, f)_{\mathcal{T}_h}
  \quad \forall \boldsymbol{v}_h \in \boldsymbol{X}_h,
\end{equation}
where
\begin{multline}
  a_h(\boldsymbol{w}, \boldsymbol{v})
  :=
  (\nabla w, \nabla v)_{\mathcal{T}_h}
  + \langle \alpha h_K^{-1}(w - \bar{w}), v-\bar{v} \rangle_{\partial\mathcal{T}_h}
  \\
  - \langle w-\bar{w}, \nabla v \cdot \boldsymbol{n} \rangle_{\partial\mathcal{T}_h}
  - \langle v-\bar{v}, \nabla w \cdot \boldsymbol{n} \rangle_{\partial\mathcal{T}_h}.
\end{multline}
Here $\alpha$ is a penalty parameter that needs to be chosen
sufficiently large \cite{Wells:2011}.

%------------------------------------------------------------------------------
\subsection{Static condensation}

A feature of the EDG and HDG methods is that local (element)
degrees-of-freedom can be eliminated from the discretization. For
higher-order accurate discretizations this static condensation can
significantly reduce the size of the problem. To obtain the reduced
problem we define the function $v_h^L(\bar{m}_h, s) \in V_h$ such that
its restriction to the element $K$ satisfies: given
$s \in L^2(\Omega)$ and $\bar{m}_h \in \bar{X}_h$,
\begin{multline}
  (\nabla v_h^L, \nabla w_h)_{K}
  + \langle \alpha h_K^{-1}v_h^L, w_h \rangle_{\partial K}
  - \langle v_h^L, \nabla w_h \cdot \boldsymbol{n} \rangle_{\partial K}
  - \langle w_h, \nabla v_h^L \cdot \boldsymbol{n} \rangle_{\partial K}
  \\
  =
  (w_h, s)_{K}
  + \langle \alpha h_K^{-1}\bar{m}_h, w_h \rangle_{\partial K}
  - \langle \bar{m}_h, \nabla w_h \cdot \boldsymbol{n} \rangle_{\partial K},
\end{multline}
for all $w_h \in Q^k(K)$. If $\boldsymbol{u}_h \in \boldsymbol{X}_h$
satisfies \cref{eq:edghdg}, then $u_h = u_h^f + l(\bar{u}_h)$ where
$u_h^f := v_h^L(0, f)$ and $l(\bar{u}_h) :=
v_h^L(\bar{u}_h,0)$. Furthermore, $\bar{u}_h \in \bar{X}_h$ satisfies
 \cite{Cockburn:2009,Rhebergen:2018b}:
\begin{equation}
  \label{eq:statcondens}
  \bar{a}_h(\bar{u}_h, \bar{v}_h) = (l(\bar{v}_h), f)_{\mathcal{T}_h} \quad \forall \bar{v}_h \in \bar{X}_h
\end{equation}
where
\begin{equation}
  \bar{a}_h(\bar{u}_h, \bar{v}_h) := a_h( (l(\bar{u}_h), \bar{u}_h), (l(\bar{v}_h), \bar{v}_h) ).
\end{equation}
We remark that \cref{eq:statcondens} is the EDG or HDG method after
eliminating the element degrees-of-freedom.

It will be useful to consider also the matrix representation of the
EDG and HDG methods. For this, let ${\bf u}_h \in \mathbb{R}^{n_h}$ be the
vector of the discrete solution with respect to the basis for $V_h$
and let ${\bf \bar{u}}_h \in \mathbb{R}^{\bar{n}_h}$ be the vector of
the discrete solution with respect to the basis for $\bar{V}_h$. We
can write \cref{eq:edghdg} as
\begin{equation}
  \label{eq:matsystem}
  \begin{bmatrix}
    A & B^T \\ B & C
  \end{bmatrix}
  \begin{bmatrix}
    {\bf u}_h \\ {\bf \bar{u}}_h
  \end{bmatrix}
  =
  \begin{bmatrix}
    G_1 \\ G_2
  \end{bmatrix},
\end{equation}
where $A$, $B$, $C$ are the matrices obtained from
$a_h((0, \cdot), (0, \cdot))$, $a_h((\cdot, 0), (0, \cdot))$, and
$a_h((0, \cdot), (0, \cdot))$, respectively. Since $A$ is a block
diagonal matrix it is cheap to compute its inverse. Then, using
${\bf u}_h = A^{-1}(G_1 - B^T{\bf \bar{u}}_h)$ we eliminate ${\bf u}_h$
from \cref{eq:matsystem} and find:
\begin{equation}
  \label{eq:matsystemred}
  (-BA^{-1}B^T+C){\bf \bar{u}}_h = G_2 - BA^{-1}G_1
  \qquad \leftrightarrow \qquad K_h{\bf \bar{u}}_h = f_h.
\end{equation}
This trace system is the matrix representation of
\cref{eq:statcondens}. In the remainder of this work we will present
and analyze geometric multigrid with Vanka relaxation for the solution
of \cref{eq:matsystemred}.

%------------------------------------------------------------------------------
\section{Geometric multigrid method}
\label{sec:two-grid}

The geometric multigrid algorithm consists of: (1) applying
pre-relaxation on the fine grid; (2) a coarse-grid correction step in
which the residual is restricted to a coarse grid, a coarse-grid
problem is solved (either exactly or by applying multigrid
recursively), interpolating the resulting solution as an error
correction to the fine grid approximation; and (3) applying
post-relaxation.

In this section we present the different operators in a geometric
multigrid method for the solution of the trace system
\cref{eq:matsystemred}. To set up notation, let $\mathcal{T}_{n,h}$ be
a finite sequence of increasingly coarser meshes with
$n=1,2,3,\hdots,N$. For $1 \le n < m \le N$ we denote the restriction
operator by $R_{n,h}^{m,h}:\mathcal{T}_{n,h} \to \mathcal{T}_{m,h}$, the
prolongation operator by
$P_{m,h}^{n,h}:\mathcal{T}_{m,h} \to \mathcal{T}_{n,h}$, and the
coarse-grid operator by $K_{m,h}$.

%------------------------------------------------------------------------------
\subsection{Relaxation scheme}
\label{subsec:relaxation-scheme}

Many different relaxation methods may be used in multigrid
algorithms. In our analysis we consider additive Vanka type relaxation
(block Jacobi relaxation defined by Vanka patches) and compare its
performance to the classical relaxation iterations of pointwise Jacobi
and Gauss--Seidel. In this section we introduce additive Vanka
relaxation relaxation \cite{Frommer:1999Vanka, Saad2003book} following
the description in \cite{farrell2019local}.

Let $\mathcal{D}$ denote the set of DOFs of ${\bf \bar{u}}_h$ and let
$\mathcal{D}_i$, $i=1, \hdots, J$, be subsets of unknowns with
$\mathcal{D}=\cup_{i=1}^{J} \mathcal{D}_i$. Let $V_i$ be the
restriction operator mapping from vectors over the set of all
unknowns, $\mathcal{D}$, to vectors whose unknowns consist of the DOFs
in $\mathcal{D}_{i}$. Then $K_i=V_iK_hV_i^T$ is the restriction of
$K_h$ to the $i$-th block of DOFs. Moreover, let
$W_i=\text{diag}(w^{i}_1,w^{i}_2,\hdots,w^{i}_{m_i})$ for
$i=1,\hdots,J$ be a diagonal weight matrix for each block $i$, where
$m_i$ is the dimension of $K_i$. Then, for a given approximation
${\bf \bar{u}}^{(j)}_h$, we solve in each Vanka block
$i = 1, \hdots, J$ the linear system
\begin{equation}
  \label{eq:vanka-block}
  K_i \delta_i =V_i(b_h-K_h{\bf \bar{u}}^{(j)}_h),
\end{equation}
and update ${\bf \bar{u}}^{(j)}_h$ according to
\begin{equation}
  \label{eq:vanka-omega}
  {\bf \bar{u}}^{(j+1)}_h ={\bf \bar{u}}^{(j)}_h +\omega \sum_{i=1}^{J}V_i^T W_i \delta_i,
\end{equation}
where $\omega$ is a tunable parameter. For the rest of this paper it
is useful to note that the error-propagation operator of the additive
Vanka relaxation scheme is given by
\begin{equation}
  \label{eq:vanka-smoother}
  S_h = I - \omega M_h^{-1} K_h,
\end{equation}
where
\begin{equation*}
  M_h^{-1} = \sum_{i=1}^{J}V_i^{T}W_i K_i^{-1}V_i.
\end{equation*}

Depending on the discretization method (EDG, HDG, and CG), the sets
$\mathcal{D}_i$ $i=1, \hdots, J$ are chosen differently. However, for
all discretization methods we consider two classes of determining
$\mathcal{D}_i$, namely, via vertex-wise patches and via element-wise
patches. Vertex-wise patches $\mathcal{D}_i$ consist of the DOFs on
the vertex $\boldsymbol{ v}_i$, the DOFs on the interior of all edges
that share vertex $\boldsymbol{v}_i$, and any DOFs on the interiors of
the elements that contains vertex $\boldsymbol{v}_i$. On element-wise
patches $\mathcal{D}_i$ consists of all DOFs on the $i^{\text{th}}$
element and its boundary. As an example we plot the different patches
for CG for $k=2$, EDG for $k=2$, and HDG for $k=1$ in
\cref{fig:Vanka-patches-plot}. Note that these two Vanka-type patches
are applicable also on unstructured meshes.

\begin{figure}[tbp]
  \centering
  \subfloat[Vertex-wise Vanka block. \label{fig:CG-Vanka-patch-plot_a}]
  {\includegraphics[width=0.4\textwidth]{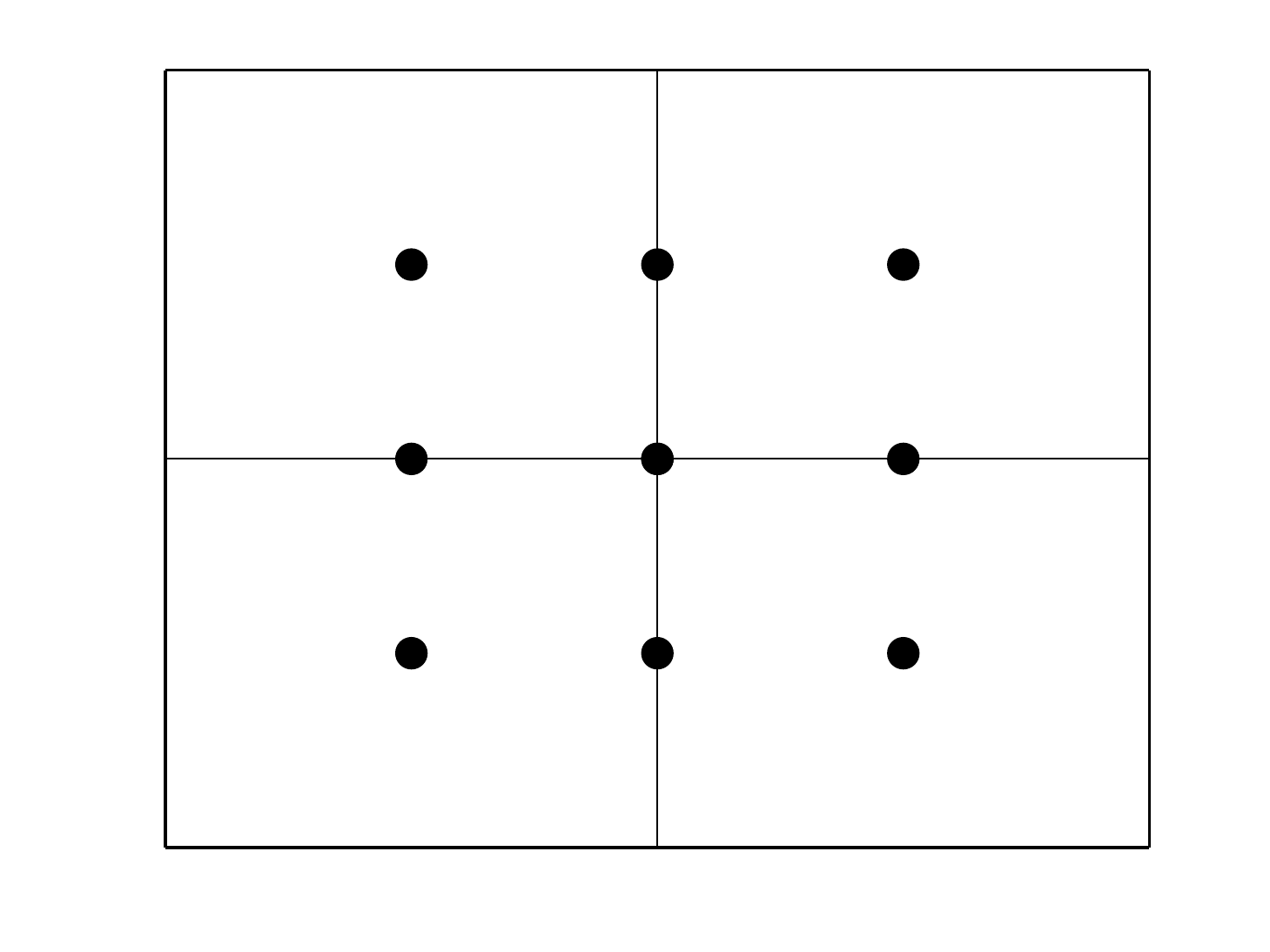}}
  \subfloat[Element-wise Vanka block. \label{fig:CG-Vanka-patch-plot_b}]
  {\includegraphics[width=0.4\textwidth]{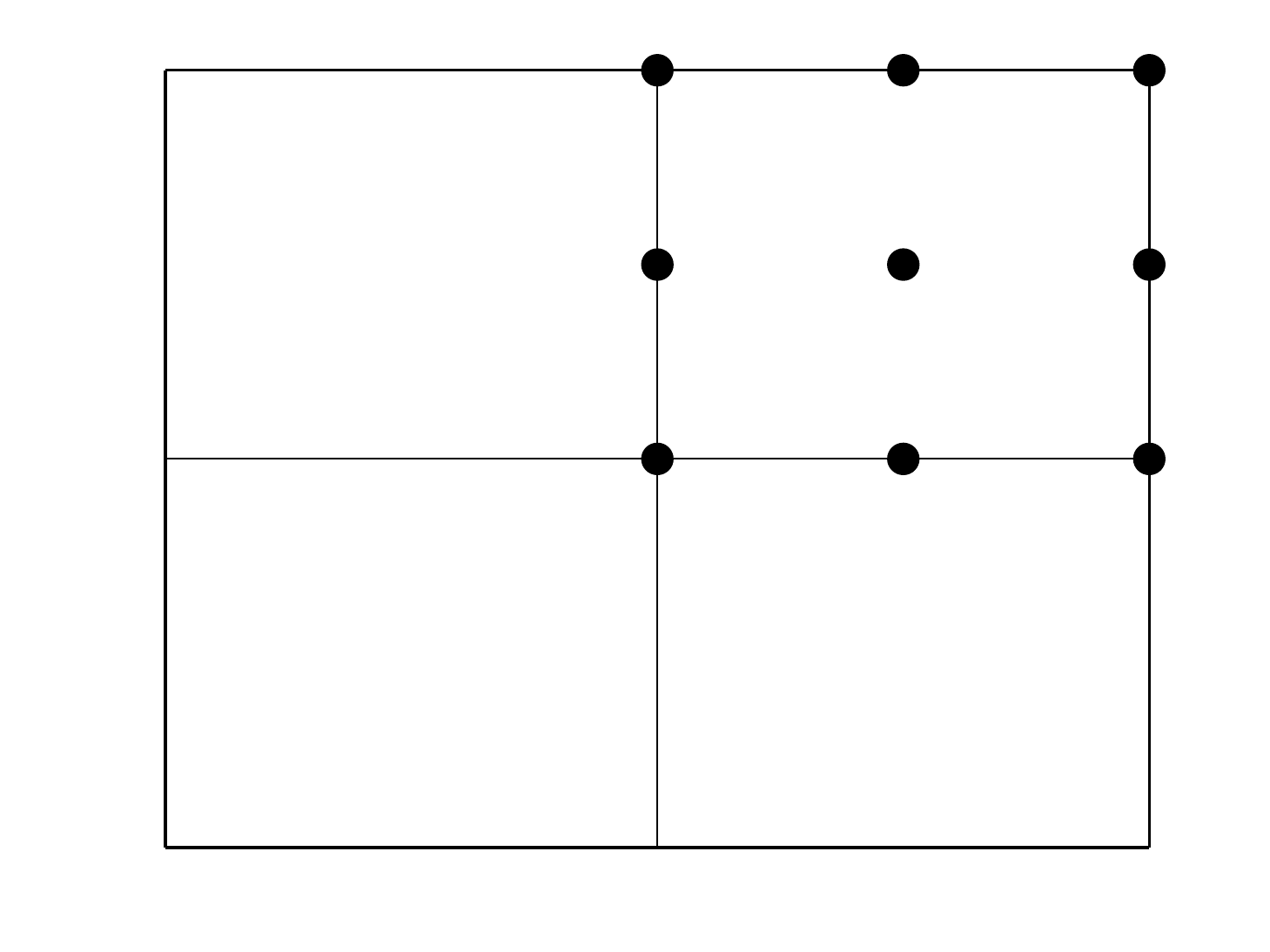}}
  \\
  \subfloat[Vertex-wise Vanka block. \label{fig:EDG-Vanka-patch-plot_a}]
  {\includegraphics[width=0.4\textwidth]{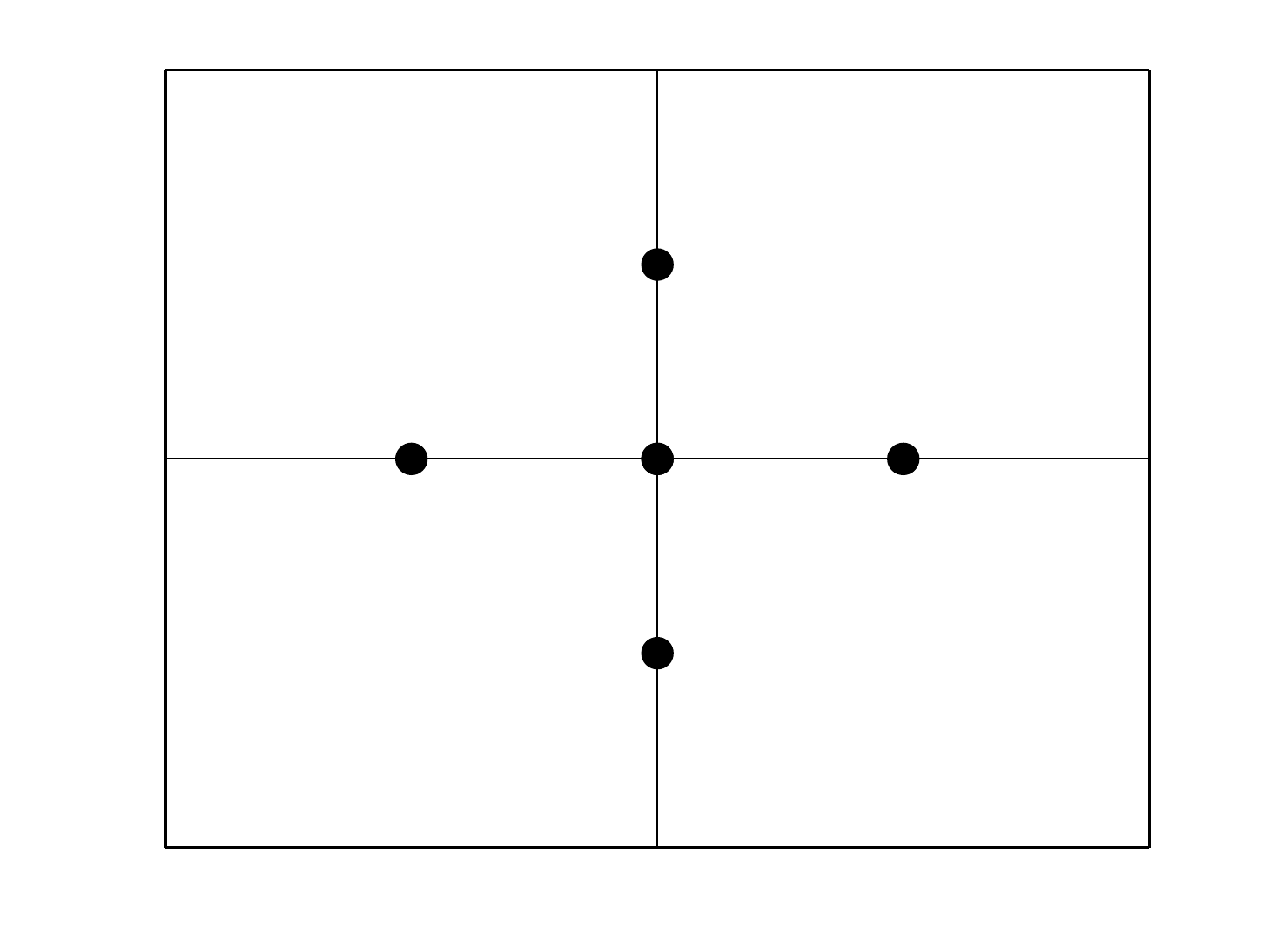}}
  \subfloat[Element-wise Vanka block. \label{fig:EDG-Vanka-patch-plot_b}]
  {\includegraphics[width=0.4\textwidth]{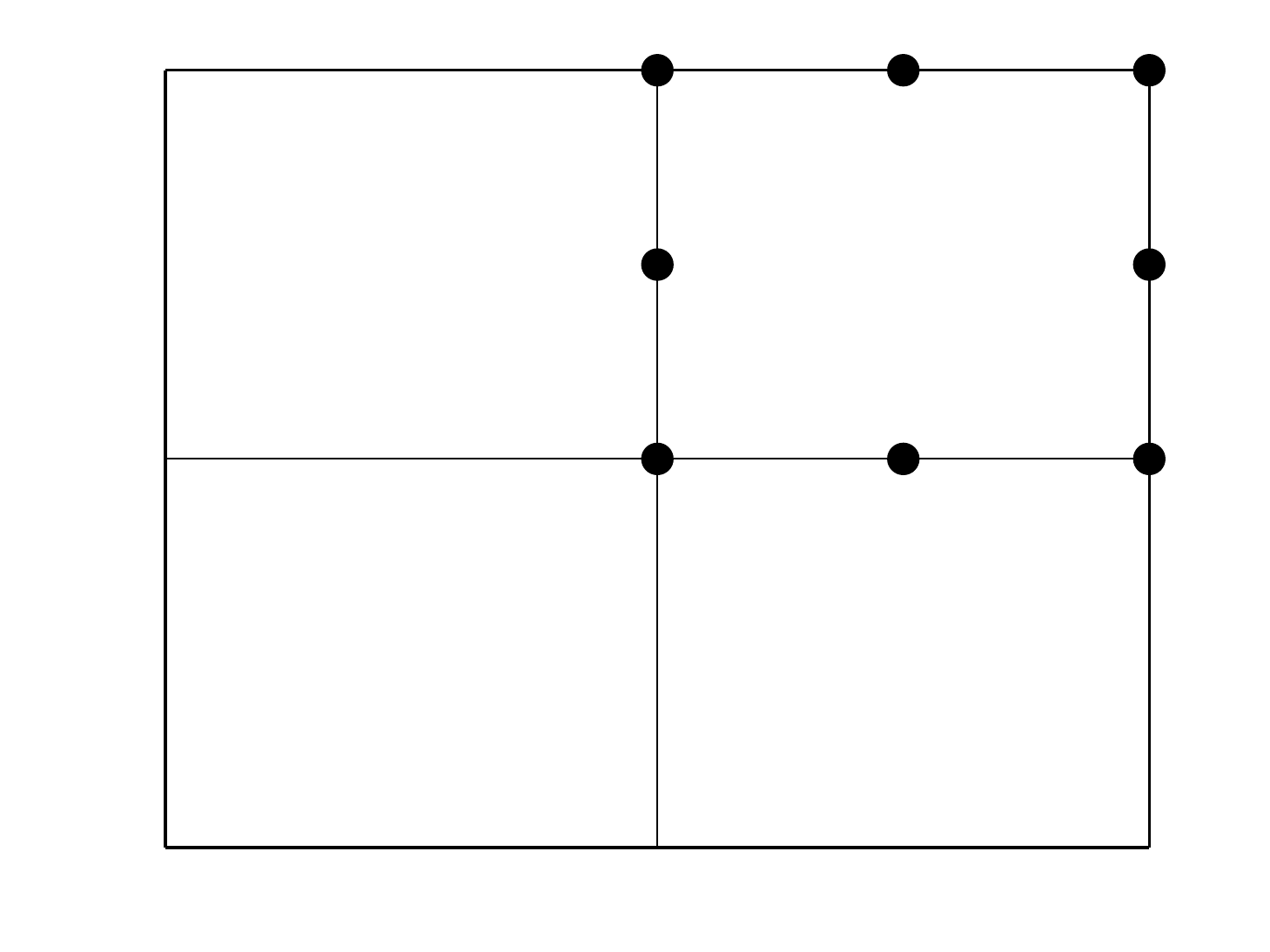}}
  \\
  \subfloat[Vertex-wise Vanka block. \label{fig:HDG-Vanka-patch-plot_a}]
  {\includegraphics[width=0.4\textwidth]{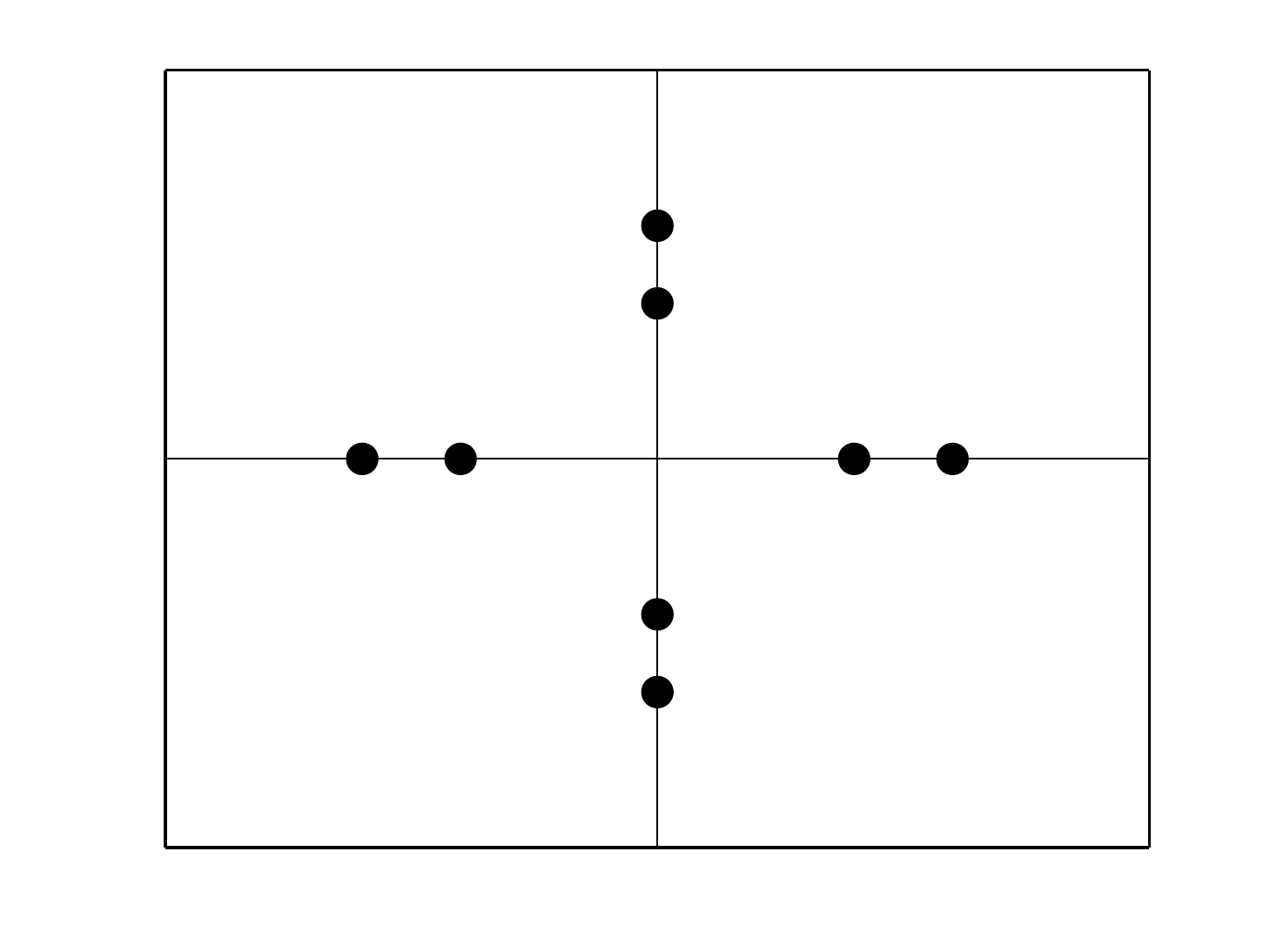}}
  \subfloat[Element-wise Vanka block. \label{fig:HDG-Vanka-patch-plot_b}]
  {\includegraphics[width=0.4\textwidth]{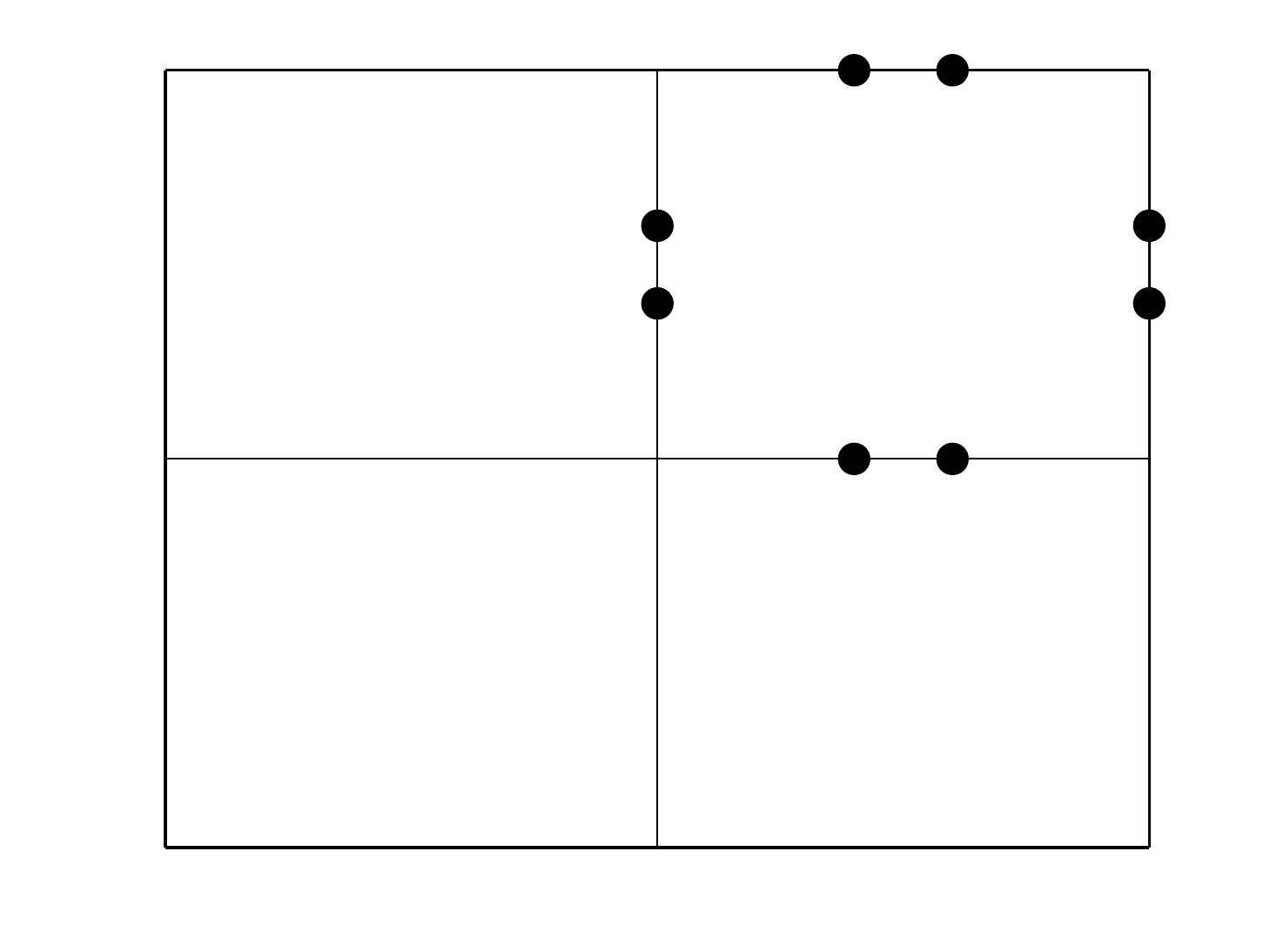}}
  \caption{Vertex- and element-wise Vanka patches. Top row: CG
    ($k=2$). Middle row: EDG ($k=2$). Bottom row: HDG ($k=1$).}
  \label{fig:Vanka-patches-plot}
\end{figure}

Given the set $\mathcal{D}_i$ we next describe the weight matrix
$W_i = \text{diag}(w_1^i,\hdots,w_{m_i}^i)$. Here $w_k^i$ is the
reciprocal of the number of patches that contain DOF $k$. For example,
consider the case of Continuous Galerkin with $k=2$ and a vertex-wise
Vanka block (see \cref{fig:CG-Vanka-patch-plot_a}). The vertex DOF is
not shared by other patches, therefore its weight is 1. The edge
degrees-of-freedom are shared by two patches, therefore their weight
is 1/2. The DOFs on the interior of an element are shared by four
patches, therefore their weight is 1/4. For this example we
furthermore note that $K_i$ is a $9 \times 9$ matrix and, if $K_h$ is
an $n\times n$ matrix, $V_i$ is a $9 \times n$ matrix.

We end this section by noting that the size of $K_i$ increases with
increasing degree of the polynomial approximation used in the
discretization. To apply the additive Vanka smoother
\cref{eq:vanka-smoother} it is necessary to compute the inverse of
$K_i$. To reduce the cost of computing $K_i^{-1}$ we therefore also
consider a lower-triangular approximation to $K_i$. We refer to this
as Lower-Triangular-Vertex-Wise and Lower-Triangular-Element-Wise
Vanka relaxation.

%------------------------------------------------------------------------------
\subsection{Grid-transfer operators and the coarse-grid operator}

While standard finite-element interpolation as in
\cite{Sampath2010FEM} can be used for a Continuous Galerkin finite
element method, this approach is not possible for hybridized methods
since $\bar{X}_{m,h} \nsubseteq \bar{X}_{n,h}$ with $m=n+1$. We
therefore use the Dirichlet-to-Neumann (DtN) interpolation from
\cite{Wildey:2019} which we describe next for the HDG and EDG methods.

In what follows, we assume $m=n+1$.  First we decompose the set of all
fine-level faces in $\mathcal{T}_{n,h}$ as
$\mathcal{F}_{n,h} = \mathcal{F}_{n,h}^I \oplus \mathcal{F}_{n,h}^B$,
with $\mathcal{F}_{n,h}^B$ the set of all faces in $\mathcal{F}_{n,h}$
that are in the coarse-level face set $\mathcal{F}_{m,h}$ (they form
the boundaries of the coarse elements), and with $\mathcal{F}_{n,h}^I$
the set of all faces in $\mathcal{F}_{n,h}$ that are not in
$\mathcal{F}_{m,h}$ (they lie in the interior of the coarse
elements). The idea is then to split the set of degrees-of-freedom of
${\bf \bar{u}}_h$ into two groups
$\mathcal{D} =\mathcal{D}_I \cup \mathcal{D}_B$, where $\mathcal{D}_B$
is the set of DOFs located on the edges in $\mathcal{F}_{n,h}^B$ and
$\mathcal{D}_I$ is the set of DOFs located on the edges in
$\mathcal{F}_{n,h}^I$. We illustrate this for the HDG method with
$k=1$ on a rectangular mesh in \cref{fig:HDG-FC-grid-plot}. Then
denoting by $\bar{X}_h^B$ the part of $\bar{X}_h$ corresponding to
$\mathcal{D}_B$, it is clear that
$\bar{X}_{m,h} \subset \bar{X}_{n,h}^B$.

We adapt \cite{Sampath2010FEM} to define the first part of the
prolongation operator mapping, from $\bar{X}_{m,h}$ to
$\bar{X}_{n,h}^B$. In particular, we define
$P_B : \bar{X}_{m,h} \to \bar{X}_{n,h}^B$, with $m=n+1$, as
\begin{equation}
  \sum_{F\in\mathcal{F}_{n,h}^B}
  \langle P_B\bar{v}, \bar{u} \rangle_{F}
  =
  \sum_{F\in\mathcal{F}_{n,h}^B}
  \langle \bar{v}, \bar{u} \rangle_{F}
  \quad
  \forall \bar{v} \in \bar{X}_{m,h},\ \bar{u} \in \bar{X}_{n,h}^B.
\end{equation}
For the multigrid method, however, we require a prolongation operator
$P_{m,h}^{n,h}$ mapping from $\bar{X}_{m,h}$ to the whole of
$\bar{X}_{n,h}$, which we discuss next.

\begin{figure}[tbp]
  \centering
  \subfloat[Coarse grid. \label{fig:HDG-FC-grid-plot_a}]
  {\includegraphics[width=0.45\textwidth]{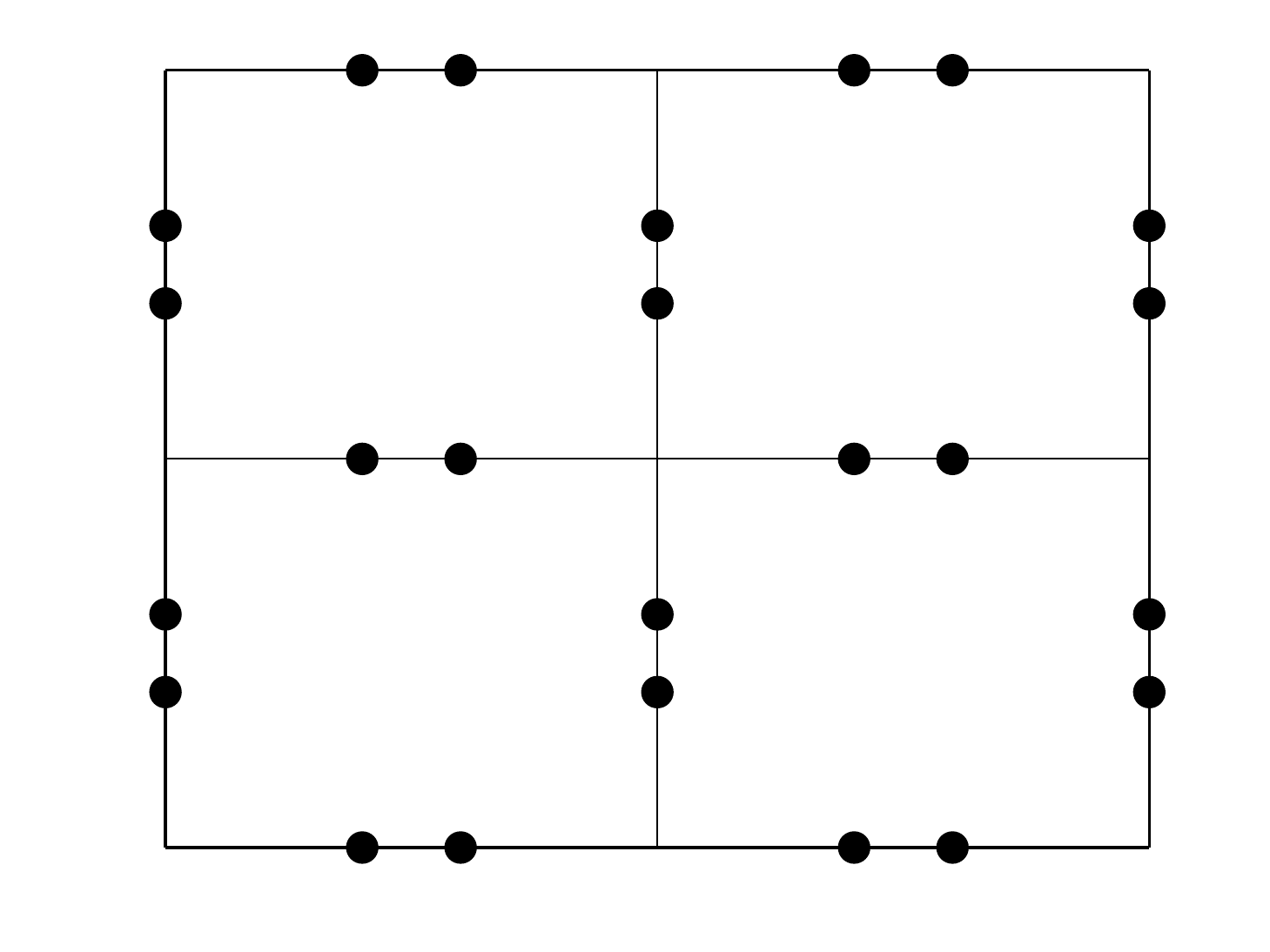}}
  \subfloat[Fine grid. \label{fig:HDG-FC-grid-plot_b}]
  {\includegraphics[width=0.45\textwidth]{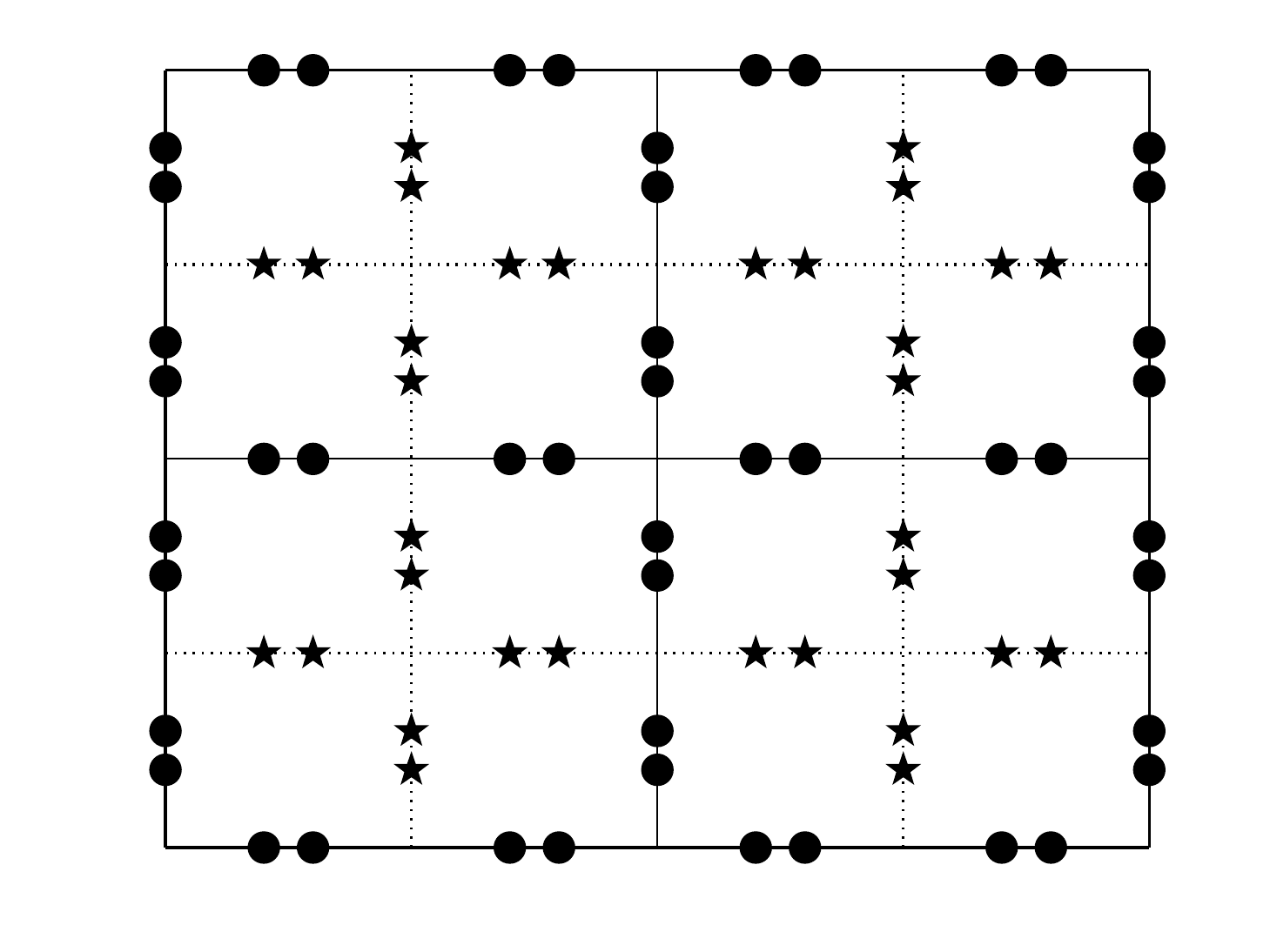}}
  \caption{Coarse and fine grids for HDG with $k=1$. The filled
    circles are DOFs in $\mathcal{D}_B$ and the stars are DOFs in
    $\mathcal{D}_I$.}
  \label{fig:HDG-FC-grid-plot}
\end{figure}

Splitting the vector ${\bf \bar{u}}_h$ in \cref{eq:matsystemred} into
DOFs in $\mathcal{D}_I$ and $\mathcal{D}_B$ we can write
\cref{eq:matsystemred} as
\begin{equation*}
  \begin{bmatrix}
    K_{II} & K_{IB}\\ K_{BI} & K_{BB}
  \end{bmatrix}
  \begin{bmatrix}
    {\bf \bar{u}}_{I}\\ {\bf \bar{u}}_{B}
  \end{bmatrix}
    = f_h.
\end{equation*}
We can then define the DtN prolongation operator as
\begin{equation}
  \label{eq:DtN-map}
  P_{m,h}^{n,h} =
  \begin{bmatrix}
    P_I \\ P_B
  \end{bmatrix}
  =
  \begin{bmatrix}
    -K^{-1}_{II} K_{IB} P_B \\ P_B
  \end{bmatrix}.
\end{equation}
We remark that this operator can be computed locally. We furthermore
remark that it was shown in \cite{Wildey:2019} that this DtN
prolongation operator preserves energy when transferring information
between two levels, i.e., for $m=n+1$,
\begin{equation*}
  \bar{a}_{n,h}(P_{m,h}^{n,h}\bar{v}_{m,h}, P_{m,h}^{n,h}\bar{v}_{m,h})
  =
  \bar{a}_{m,h}(\bar{v}_{m,h}, \bar{v}_{m,h}) \quad \forall \bar{v}_{m,h} \in \bar{X}_{m,h}.
\end{equation*}

Given the prolongation operator we then define the restriction
operator as $R_{n,h}^{m,h} = (P_{m,h}^{n,h})^T$ and use the Galerkin
approximation of $K_{n,h}$ as our coarse-grid operator, i.e.,
$K_{m,h} = R_{n,h}^{m,h}K_{n,h}P_{m,h}^{n,h}$.

%------------------------------------------------------------------------------
\section{A local Fourier analysis framework for HDG, EDG, and CG}
\label{sec:LFA-tg}

Let $\mathcal{T}_h$ denote a fine mesh of $\Omega$ and
$\mathcal{T}_{H}$ denote a coarse mesh of $\Omega$ such that
$H=2h$. From now on, all meshes we consider will be Cartesian. The
two-grid error-propagation operator is given by
\begin{equation}
  \label{TG-error-form}
  E_h = S_h^{\nu_2}(I_h - P_H^h(K_H)^{-1}R_h^HK_h)S_h^{\nu_1},
\end{equation}
where $\nu_1$, $\nu_2$ are the number of pre- and post relaxation
sweeps, respectively, and $I_h$ denotes the identity operator. To
analyze the two-grid error-propagation operator, and hence to obtain a
measure of the efficiency of the two-grid method applied to HDG, EDG,
and CG discretizations of the Poisson problem, we use Local Fourier
Analysis (LFA) \cite{Trottenberg:book,Wienands:2005}.

LFA was introduced in \cite{Brandt1977LFA} to study the convergence
behavior of multigrid methods for boundary value problems. Assume
$L_h$ is a discrete operator obtained by discretizing a PDE on an
infinite two dimensional domain.  $L_h$ can be thought of as a matrix
of infinite size, but we represent it by operators that operate on the
DOFs near a generic grid point and that are specified by
two-dimensional stencils that we assume have constant stencil
coefficients. The eigenfunctions of $L_h$ can be expressed by discrete
Fourier modes, resulting in a representation of $L_h$ by an
$r\times r$ matrix,
$\tilde{L}_h(\boldsymbol{\theta}) \in \mathbb{C}^{r \times r}$ where
$\boldsymbol{\theta} \in (-\pi/2, 3\pi/2]^2$ and $r \ge 1$ is small
and depends on the discretization. This
$\tilde{L}_h(\boldsymbol{\theta})$ is called the symbol of $L_h$.
Specifically, for a scalar PDE, when there is a single degree of
freedom located on the mesh of a cartesian grid, then $r=1$ and the
symbol is a scalar.  However, when there are different degrees of
freedom that may be located at different locations on the grid
(e.g. nodes, edges, centers, etc.), then $r>1$ equals the number of
different degrees of freedom, and the symbol is matrix-valued.

Fundamental to LFA is that the properties of $L_h$ can be described by
the small matrix $\tilde{L}_h(\boldsymbol{\theta})$. For example, the
efficiency of multigrid on a finite grid is measured by the spectral
radius of the two-grid error-propagation matrix $E_h$
\cref{TG-error-form}. However, since $E_h$ is typically very large, we
will find instead the spectral radius of the symbol of the two-grid
error-propagation operator corresponding to the extension of $E_h$ to
the infinite domain. In contrast to node-based discretization problems
with one DOF per node, we consider here discretizations with multiple
DOFs per node, edge, and element. The key to LFA is the identification
of the eigenspace of $K_h$ and the symbol of $K_h$, and finding the
LFA representation of $E_h$ defined by \cref{TG-error-form}. This will
be done in this section. The methodology used in our LFA analysis for
high-order hybridized and embedded discontinuous Galerkin
discretizations of the Laplacian is similar to the methods used in the
recent papers \cite{Boonen2008curl, de2020convergence,
  farrell2019local, He:2019, hetwo, Maclachlan:2011, Rodrigo:2015}.
The theory for the invariant space of edge-based operators was first
given in \cite{Boonen2008curl}, for the curl-curl equations.  While
describing the general principles of the approach in the next
sections, we extend the theory of \cite{Boonen2008curl} to the case
where we have multiple degrees of freedom located on multiple
different grid locations, including vertical and horizontal edges,
nodes and cell centers of the grid.

%------------------------------------------------------------------------------
\subsection{Infinite grid}

For our analysis we follow a similar approach as \cite{Boonen2008curl}
by first defining an appropriate infinite grid and subgrids. To define
these subgrids, we first lump all the $X$-type DOFs on a horizontal
edge and $Y$-type DOFs on a vertical edge to the midpoint of that
edge. All $C$-type DOFs will be lumped to the center of an element and
$N$-type DOFs are located at a node. We then consider the following
two-dimensional infinite uniform grid
$\boldsymbol{G}_h=\cup_{\alpha\in\{N,X,Y,C\}}\,
\boldsymbol{G}_h^{\alpha}$, with subgrids
\begin{equation}
  \label{four-mesh-types}
  \bm{G}^{\alpha}_{h}
  =
  \cbr[1]{ \boldsymbol{x}^{\alpha} := (x^{\alpha}_1,x^{\alpha}_2) = (k_{1},k_{2})h + \sigma^{\alpha},\ (k_1,k_2) \in \mathbb{Z}^2},
\end{equation}
where
\begin{equation*}
  \sigma^{\alpha} =
  \begin{cases}
    (0,0)     & \text{if } \alpha=N,\\
    (h/2,0)   & \text{if } \alpha=X,\\
    (0,h/2)   & \text{if } \alpha=Y,\\
    (h/2,h/2) & \text{if } \alpha=C.
  \end{cases}
\end{equation*}
We will refer to $\boldsymbol{G}^{N}_{h}$ as $N$-type points on the
grid $\boldsymbol{G}_h$, associated with $N$-type DOFs. Furthermore,
$\boldsymbol{G}^{X}_{h}$, $\boldsymbol{G}^{Y}_{h}$, and
$\boldsymbol{G}^{C}_{h}$ will be referred to as $X$-, $Y$-, and
$C$-type points on the grid $\boldsymbol{G}_h$. These are associated,
respectively, with $X$-, $Y$-, and $C$-type DOFs.  We remark that it
is possible that there is more than one DOF at a particular
location. We will use subscripts to distinguish between them. For
example, $X_1$ and $X_2$ are two $X$-type DOFs on a horizontal
edge. The coarse grids $\boldsymbol{G}_H$ are defined similarly.

\begin{remark}
  The CG method for $k>1$ consists of $N$-, $X$-, $Y$-, and $C$-type
  DOFs and therefore requires all four subgrids
  $\boldsymbol{G}_h^{\alpha}$, $\alpha =N,X,Y,C$. We refer to
  \cite{hetwo} for the case $k=2$. For $k=1$ CG only has $N$-type DOFs
  and therefore requires only the $\boldsymbol{G}_h^N$ grid. EDG is
  identical to CG for $k=1$. For the EDG method for $k>1$ the $C$-type
  DOFs have been eliminated by static condensation and therefore
  requires only the subgrids $\boldsymbol{G}_h^{\alpha}$,
  $\alpha=N,X,Y$. For $k \ge 1$ the HDG method only has $X$- and
  $Y$-type DOFs and so requires the subgrids $\boldsymbol{G}_h^X$ and
  $\boldsymbol{G}_h^Y$.
\end{remark}

%------------------------------------------------------------------------------
\subsection{Partitioning of the discrete operator}
\label{ss:partitioningdiscop}

Let $K_h$ be the EDG/HDG matrix given by \cref{eq:matsystemred} or the
matrix obtained from a continuous Galerkin discretization of the
Laplacian defined on the mesh $\boldsymbol{G}_h$. We will treat $K_h$
as an operator on the infinite mesh $\boldsymbol{G}_h$. To take into
account that the DOFs on $\boldsymbol{G}_h^N$, $\boldsymbol{G}_h^X$,
$\boldsymbol{G}_h^Y$, and $\boldsymbol{G}_h^C$ are different, we
partition the operator $K_h$ according to the different groups of
DOFs:
\begin{equation}
  \label{eq:Lhpartitioning}
  K_h =
  \begin{bmatrix}
    K_{NN} &K_{NX}     &K_{NY}  &K_{NC} \\
    K_{XN} &K_{XX}     &K_{XY}  &K_{XC} \\
    K_{YN} &K_{YX}     &K_{YY}  &K_{YC} \\
    K_{CN} &K_{CX}     &K_{CY}  &K_{CC}
  \end{bmatrix}.
\end{equation}
When thinking about $K_h$ as an infinite matrix, this corresponds to
ordering the grid points in $\boldsymbol{G}_h$ in the order of the
subgrids $\boldsymbol{G}_h^{\alpha}$, $\alpha\in\{N,X,Y,C\}$, and, for
example, operator $K_{NC}$ is a mapping from a grid function on
$C$-type points to a grid function on $N$-type points.

Note furthermore that the EDG method does not include the
$C$-type DOFs and the HDG method does not include the $C$- and
$N$-type DOFs.
As an example, consider the HDG method. Then we write
\begin{equation}
  K_h =
  \begin{bmatrix}
    K_{XX} & K_{XY} \\
    K_{YX} & K_{YY}
  \end{bmatrix}.
\end{equation}
If, furthermore, $k=1$ in HDG (so that there are two DOFs per edge),
then we can again partition the submatrix $K_{\alpha\beta}$,
$\alpha,\beta \in \cbr[0]{X,Y}$ into a $2 \times 2$-block matrix,
which is given by
\begin{equation}
  \label{eq:secd-partition}
  K_{\alpha\beta} =
  \begin{bmatrix}
    K_{\alpha_1\beta_1} & K_{\alpha_1\beta_2} \\
    K_{\alpha_2\beta_1} & K_{\alpha_2\beta_2}
  \end{bmatrix},
\end{equation}
where, for example, the operator $K_{X_1Y_2}$ maps from the second
degree of freedom on the $Y$-edges to the first degree of freedom on
the $X$-edges.

In the general case we use the short-hand notation
\begin{equation}
  \label{eq:block-matrix-to-operator}
  K_h =(K_{\alpha_i\beta_j}), \quad \text{with}\quad
  \alpha,\beta \in\{N,X,Y,C\},\
  i=1,\cdots, r_{\alpha},\
  j=1,\cdots,r_{\beta},
\end{equation}
where $r_{\alpha}$ is the number of DOFs on a single edge (if
$\alpha = X$ or $\alpha = Y$), in a single node (if $\alpha = N$), or
on a single element (if $\alpha = C$). Let $r=r_N+r_X+r_Y+r_C$ be the
total number of DOFs per element.

Let $w_h(\boldsymbol{x})$ be a grid function on $\boldsymbol{G}_h$,
and by $w_h(\boldsymbol{x}^\alpha)$ we mean the grid function
$w_h(\boldsymbol{x})$ restricted to grid points
$\boldsymbol{x}^\alpha \in \bm{G}^{\alpha}_{h} \subset \bm{G}_{h}$.
Consider operator $K_{\alpha_i\beta_j}$ from the $jth$ degree of
freedom on grid $\bm{G}^{\beta}_{h}$ to the $ith$ degree of freedom on
grid $\bm{G}^{\alpha}_{h}$.  The action of linear operator
$K_{\alpha_i\beta_j}$ on grid function $w_h(\boldsymbol{x})$ is given
by
\begin{equation}
  \label{defi-stencil-elements}
  K_{\alpha_i\beta_j}w_{h}(\boldsymbol{x}^\alpha) =
  \sum_{\boldsymbol{\kappa}\in\boldsymbol{V}_{\alpha_i\beta_j}} s^{(\alpha_i,\beta_j)}_{\boldsymbol{\kappa}}
  w_{h}(\boldsymbol{x}^\alpha+\boldsymbol{\kappa}h),
\end{equation}
where the (constant) coefficients
$s_{\boldsymbol{\kappa}}^{(\alpha_i,\beta_j)}$ define the stencil
representation of $K_{\alpha_i\beta_j}$ as
\begin{equation}
  \label{defi-symbol-operator}
  K_{\alpha_i\beta_j}  := [s_{\boldsymbol{\kappa}}]_{\alpha_i\beta_j}.
\end{equation}
Here we assume that only a finite number of coefficients
$s^{(\alpha_i,\beta_j)}_{\boldsymbol{\kappa}}$ are nonzero, i.e.,
$\boldsymbol{V}_{\alpha_i\beta_j}$ is a finite set of offset vectors
$\boldsymbol{\kappa}$ such that
$\boldsymbol{x}^\alpha+\boldsymbol{\kappa}h \in
\boldsymbol{G}_h^{\beta}$.  Note that our notation
$K_{\alpha_i\beta_j}w_{h}(\boldsymbol{x}^\alpha)$ in
\cref{defi-stencil-elements} emphasizes that the grid function
$K_{\alpha_i\beta_j}w_{h}$ is defined on grid $\bm{G}^{\alpha}_{h}$,
and its function values are linear combinations of values of grid
function $w_{h}(\boldsymbol{x})$ restricted to grid
$\bm{G}^{\beta}_{h}$, as expressed in the right-hand side of
\cref{defi-stencil-elements}.

We now describe the four different possible stencil types for
$K_{\alpha_i\beta_j}$. For this, let
$\boldsymbol{x} \in \boldsymbol{G}_h^{\alpha}$ with
$\alpha \in \cbr[0]{N,X,Y,C}$ and denote by $\odot$ a grid point in
$\boldsymbol{G}_h^{\alpha}$ on which the operator acts. Furthermore,
let $i$ be a fixed value in $\{1, \hdots, r_{\alpha}\}$.

{\bf Stencil type 1.} This considers the case where $\alpha_i$ and
$\beta_j$ in \cref{defi-symbol-operator} share the same grid
locations, i.e., $\beta_j=\alpha_j$. Let
$\boldsymbol{V}_{\alpha_i\alpha_j}$ be a finite set such that
$\boldsymbol{x}+\boldsymbol{\kappa}h \in \boldsymbol{G}_h^{\alpha}$
for $\kappa \in \boldsymbol{V}_{\alpha_i\alpha_j}$. Then
\begin{equation*}
  [s_{\bm{\kappa}}]_{\alpha_i\alpha_j}
  = \begin{bmatrix}
            &       \vdots                   &      \vdots                     &            \vdots                    &      \\
     \cdots &  s^{(\alpha_i,\alpha_j)}_{-1,1} &    s^{(\alpha_i,\alpha_j)}_{0,1} &        s^{(\alpha_i,\alpha_j)}_{1,1}  &\cdots\\
            &  s^{(\alpha_i,\alpha_j)}_{-1,0} &    s^{(\alpha_i,\alpha_j)}_{0,0}(\odot) & s^{(\alpha_i,\alpha_j)}_{1,0}  &\\
    \cdots  &  s^{(\alpha_i,\alpha_j)}_{-1,-1}&    s^{(\alpha_i,\alpha_j)}_{0,-1}&        s^{(\alpha_i,\alpha_j)}_{1,-1} & \cdots\\
            &       \vdots                   &      \vdots                     &            \vdots                    &      \\
  \end{bmatrix},
\end{equation*}
where $s_{\kappa_1,\kappa_2}^{(\alpha_i,\alpha_j)} \in \mathbb{R}$
depends on the discretization. Note that
$s^{(\alpha_i,\alpha_j)}_{\kappa_1,\kappa_2}$ is the value of the
stencil at $\boldsymbol{x}+\boldsymbol{\kappa}h$.

{\bf Stencil type 2.} Next, let $\boldsymbol{V}_{\alpha_i\beta_j}$ be
a finite set such that
$\boldsymbol{x}+\boldsymbol{\kappa}h \in \boldsymbol{G}_h^{\beta}$
with $\kappa \in \boldsymbol{V}_{\alpha_i\beta_j}$ and
$\boldsymbol{G}_h^{\beta} = \cbr[0]{\boldsymbol{y} := \boldsymbol{x} +
  (h/2,0),\ \boldsymbol{x} \in \boldsymbol{G}_h^{\alpha}}$. Then
\begin{equation*}
  [s_{\bm{\kappa}}]_{\alpha_i\beta_j}
  = \begin{bmatrix}
            &       \vdots                            &           &     \vdots                                   &      \\
     \cdots &  s^{(\alpha_i,\beta_j)}_{-\frac{1}{2},1} &           &       s^{(\alpha_i,\beta_j)}_{\frac{1}{2},1}  &\cdots\\
            &  s^{(\alpha_i,\beta_j)}_{-\frac{1}{2},0} &    \odot  &       s^{(\alpha_i,\beta_j)}_{\frac{1}{2},0}  &       \\
    \cdots  &  s^{(\alpha_i,\beta_j)}_{-\frac{1}{2},-1}&           &       s^{(\alpha_i,\beta_j)}_{\frac{1}{2},-1} & \cdots\\
            &       \vdots                            &           &     \vdots                                   &      \\
  \end{bmatrix}.
\end{equation*}

{\bf Stencil type 3.} Now let $\boldsymbol{V}_{\alpha_i\beta_j}$ be a
finite set such that
$\boldsymbol{x}+\boldsymbol{\kappa}h \in \boldsymbol{G}_h^{\beta}$
with $\kappa\in \boldsymbol{V}_{\alpha_i\beta_j}$ and
$\boldsymbol{G}_h^{\beta} = \cbr[0]{\boldsymbol{y} := \boldsymbol{x} +
  (0,h/2),\ \boldsymbol{x} \in \boldsymbol{G}_h^{\alpha}}$.  Then
\begin{equation*}
  [s_{\bm{\kappa}}]_{\alpha_i\beta_j}
  = \begin{bmatrix}
            &       \vdots                            &      \vdots                              &     \vdots                             &      \\
     \cdots &  s^{(\alpha_i,\beta_j)}_{-1,\frac{1}{2}} &    s^{(\alpha_i,\beta_j)}_{0,\frac{1}{2}} & s^{(\alpha_i,\beta_j)}_{1,\frac{1}{2}}  &\cdots\\
            &                                         &       \odot                              &                          &\\
    \cdots  &  s^{(\alpha_i,\beta_j)}_{-1,-\frac{1}{2}}&    s^{(\alpha_i,\beta_j)}_{0,-\frac{1}{2}} & s^{(\alpha_i,\beta_j)}_{1,-\frac{1}{2}} & \cdots\\
            &       \vdots                            &      \vdots                               &     \vdots                             &      \\
  \end{bmatrix}.
\end{equation*}

{\bf Stencil type 4.} Finally, let $\boldsymbol{V}_{\alpha_i\beta_j}$
be a finite set such that
$\boldsymbol{x}+\boldsymbol{\kappa}h \in \boldsymbol{G}_h^{\beta}$
with $\kappa \in \boldsymbol{V}_{\alpha_i\beta_j}$ and
$\boldsymbol{G}_h^{\beta} = \cbr[0]{\boldsymbol{y} := \boldsymbol{x} +
  (h/2, h/2),\ \boldsymbol{x} \in \boldsymbol{G}_h^{\alpha}}$.  Then

\begin{equation*}
  [s_{\bm{\kappa}}]_{\alpha_i\beta_j}
  = \begin{bmatrix}
            &       \vdots                            &           &     \vdots                 &      \\
     \cdots &  s^{(\alpha_i,\beta_j)}_{-\frac{1}{2},\frac{1}{2}}   &                           & s^{(\alpha_i,\beta_j)}_{\frac{1}{2},\frac{1}{2}}  &\cdots\\
            &                                         &        \odot&                          &\\
    \cdots  &  s^{(\alpha_i,\beta_j)}_{-\frac{1}{2},-\frac{1}{2}}&                           & s^{(\alpha_i,\beta_j)}_{\frac{1}{2},-\frac{1}{2}} & \cdots\\
            &       \vdots                            &                                         &     \vdots                             &      \\
  \end{bmatrix}.
\end{equation*}
We remark that not all stencil types are used for all
discretizations. For example, the HDG method only uses stencil types 1
and 4, and EDG and CG use all 4 stencil types. We present an example
of the stencil notation \cref{defi-symbol-operator} for the HDG method
with $k=1$ in \cref{ap:example_hdg_k1}.

%------------------------------------------------------------------------------
\subsection{Properties of the symbol of $K_h$}
\label{sub:invariant-space}

In this section we will determine the symbol of $K_h$.  Note that
$K_h$ is a block operator with $r^2$ blocks that are each
characterized by a stencil as defined in
\cref{defi-symbol-operator,defi-stencil-elements}. We will follow a
similar approach as presented in \cite{Boonen2008curl} to account for
$K_h$ acting on different groups of DOFs, see
Eq. \cref{eq:Lhpartitioning}. Our aim is to characterize the
eigenfunctions of $K_h$ in terms of the Fourier modes
\begin{equation}\label{eq:-single-Fourier-mode}
  \varphi_h(\boldsymbol{\theta}; \boldsymbol{x}) = e^{\iota
  \boldsymbol{\theta} \cdot \boldsymbol{x} /h},
\end{equation}
where $\iota^2=-1$.  Each of the operators $K_{\alpha_i\beta_j}$ is a
block Toeplitz operator with Toeplitz blocks, and as such their
eigenfunctions are given by \cref{eq:-single-Fourier-mode}, but we
need to determine how these eigenfunctions combine to make up
invariant subspaces of $K_h$, taking into account the different
degrees of freedom and the different grid locations. To account for
this we redefine the grid.

Let $\alpha \in \{N,X,Y,C\}$ and let there be $r_{\alpha}$ DOFs,
denoted by $\alpha_1, \hdots, \alpha_{r_{\alpha}}$, located at
location $\alpha$. We define $r_{\alpha}$ copies of
$\boldsymbol{G}_h^{\alpha}$, i.e.,
$\boldsymbol{G}^{\alpha_1}_h = \hdots =
\boldsymbol{G}^{\alpha_{r_{\alpha}}}_h = \boldsymbol{G}^{\alpha}_h$.
We then define the `extended' grid $\boldsymbol{\widetilde{G}}_h$ in
which we order the grid points as follows: all grid points in
$\boldsymbol{G}_h^{N_1}$ followed by the grid points in
$\boldsymbol{G}_h^{N_2},\hdots,\boldsymbol{G}_h^{N_{r_N}}$, then the
$r_X$ copies of the grid points in $\boldsymbol{G}_h^X$, the $r_Y$
copies of the grid points in $\boldsymbol{G}_h^Y$, and the $r_C$
copies of the grid points in $\boldsymbol{G}_h^C$ (note that this
ordering is consistent with the ordering of $K_{\alpha_i\beta_j}$ in
\cref{eq:block-matrix-to-operator}). We denote the definition of the
extended grid using the $\bigotimes$ symbol:
\begin{equation}
  \boldsymbol{\widetilde{G}}_h = \bigotimes_{\alpha \in\{N,X,Y,C\}, i\in\{ 1, 2,\cdots, r_{\alpha}\}}\boldsymbol{G}^{\alpha_i}_h.
\end{equation}
We note that $\boldsymbol{\widetilde{G}}_H$ is defined similarly.

For example, for HDG with $k=1$ we write
\begin{equation}
  \boldsymbol{\widetilde{G}}_h = \boldsymbol{G}_h^{X_1} \bigotimes \boldsymbol{G}_h^{X_2} \bigotimes
  \boldsymbol{G}_h^{Y_1} \bigotimes \boldsymbol{G}_h^{Y_2}.
\end{equation}

Consider now the application of the discrete operator $K_h$ acting on
a function $w_h(\boldsymbol{x})$ defined on
$\boldsymbol{\widetilde{G}}_h$.  Considering the restriction of the
grid function $K_hw_h(\boldsymbol{x})$ to the $\alpha_i$ grid, i.e.,
evaluating $K_hw_h(\boldsymbol{x})$ in
$\boldsymbol{x}^{\alpha_i} \in \boldsymbol{G}_h^{\alpha_i}$, we obtain
\begin{align}
    K_hw_h(\boldsymbol{x}^{\alpha_i})
    &=
    \sum_{\beta\in\{N,X,Y,C\}} \sum_{j=1}^{r_{\beta}} K_{\alpha_i\beta_j}w_h(\boldsymbol{x}^{\alpha_i})
    \label{eq:Khwxgalphaia}\\
    &=
    \sum_{\beta\in\{N,X,Y,C\}} \sum_{j=1}^{r_{\beta}}\sum_{\boldsymbol{\kappa}\in \boldsymbol{V}_{\alpha_i\beta_j}} s^{(\alpha_i,\beta_j)}_{\boldsymbol{\kappa}}
    w_h(\boldsymbol{x}^{\alpha_i}+\boldsymbol{\kappa} h). \label{eq:Khwxgalphai}
\end{align}
Note that, in \cref{eq:Khwxgalphai}, the grid function
$w_h(\boldsymbol{x})$ is evaluated on the $\beta_j$ grid, with
$\boldsymbol{\kappa}\in \boldsymbol{V}_{\alpha_i\beta_j}$ and
$\boldsymbol{x}^{\alpha_i}+\boldsymbol{\kappa} h \in
\boldsymbol{G}_h^{\beta_j}$.

In particular, if we take $w_h(\boldsymbol{x})$ to be the Fourier mode
$\varphi_h(\boldsymbol{\theta};\boldsymbol{x})$ from
\cref{eq:-single-Fourier-mode} defined on all grid points
$\boldsymbol{x} \in \boldsymbol{\widetilde{G}}_h$, we obtain, for the
grid function $K_h\varphi_h(\boldsymbol{\theta};\boldsymbol{x})$
evaluated on the $\alpha_i$ grid, i.e., in
$\boldsymbol{x}^{\alpha_i} \in \boldsymbol{G}_h^{\alpha_i}$,
\begin{equation}
  \label{eq:def-K-act-on-Fourier-mode}
  \begin{split}
    K_h\varphi_h(\boldsymbol{\theta};\boldsymbol{x}^{\alpha_i})
    &=
    \sum_{\beta\in\{N,X,Y,C\}} \sum_{j=1}^{r_{\beta}} K_{\alpha_i\beta_j}\varphi_h(\boldsymbol{\theta};\boldsymbol{x}^{\alpha_i})
    \\
    &=
    \sum_{\beta\in\{N,X,Y,C\}} \sum_{j=1}^{r_{\beta}}\sum_{\boldsymbol{\kappa}\in \boldsymbol{V}_{\alpha_i\beta_j}} s^{(\alpha_i,\beta_j)}_{\boldsymbol{\kappa}}
    \varphi_h(\boldsymbol{\theta};\boldsymbol{x}^{\alpha_i}+\boldsymbol{\kappa} h)
    \\
    &=
    \sum_{\beta\in\{N,X,Y,C\}} \sum_{j=1}^{r_{\beta}}\sum_{\boldsymbol{\kappa}\in \boldsymbol{V}_{\alpha_i\beta_j}} s^{(\alpha_i,\beta_j)}_{\boldsymbol{\kappa}}
    e^{\iota\boldsymbol{\theta}\cdot \kappa} \varphi_h(\boldsymbol{\theta};\boldsymbol{x}^{\alpha_i})
    \\
    &=
    \sum_{\beta\in\{N,X,Y,C\}} \sum_{j=1}^{r_{\beta}} \widetilde{K}_{\alpha_i\beta_j}(\boldsymbol{\theta}) \varphi_h(\boldsymbol{\theta};\boldsymbol{x}^{\alpha_i}),
  \end{split}
\end{equation}
where we define
\begin{equation}
  \label{eq:classical-symbol-definition}
  \widetilde{K}_{\alpha_i\beta_j}(\boldsymbol{\theta})
  :=
  \sum_{\boldsymbol{\kappa} \in \boldsymbol{V}_{\alpha_i\beta_j}}
  s^{(\alpha_i,\beta_j)}_{\boldsymbol{\kappa}} e^{\iota \boldsymbol{\theta}\cdot\boldsymbol{\kappa}}.
\end{equation}
The scalar $\widetilde{K}_{\alpha_i\beta_j}(\boldsymbol{\theta})$ in
\cref{eq:classical-symbol-definition} is called the symbol of the
operator block $K_{\alpha_i\beta_j}$, taking into account the offset
between the grids for DOF $\alpha_i$ and DOF $\beta_j$ as encoded in
the $\boldsymbol{\kappa} \in \boldsymbol{V}_{\alpha_i\beta_j}$.

Let us now define $r=r_N+r_X+r_Y+r_C$ grid functions $\Psi^{\alpha}_{j}(\boldsymbol{\theta}; \boldsymbol{x})$ on the grid
$\widetilde{\boldsymbol{G}}_h$ by 
\begin{equation}
  \label{eq:-four-basis}
   \Psi^{\alpha}_{j}(\boldsymbol{\theta}; \boldsymbol{x})
   :=
   \begin{cases}
     \varphi_h(\boldsymbol{\theta}; \boldsymbol{x})
     & \text{on the $\alpha_j$ grid, i.e., when } \boldsymbol{x}\in\boldsymbol{G}^{\alpha_j}_h,
     \\
     0
     & \text{on all other grids},
   \end{cases}
\end{equation}
and the $r$-dimensional function space
\begin{multline}
  \label{eq:basis-Fh}
  \mathcal{F}_h(\bm{\theta}) =\text{span}\big\{\Psi_1^{N}(\bm{\theta};\cdot),\hdots,\Psi_{r_N}^{N}(\bm{\theta};\cdot),
  \Psi_1^{X}(\bm{\theta};\cdot),\hdots,\Psi_{r_X}^{X}(\bm{\theta};\cdot),
  \\
  \Psi_1^{Y}(\bm{\theta};\cdot),\hdots,\Psi_{r_Y}^{Y}(\bm{\theta};\cdot),
  \Psi_1^{C}(\bm{\theta};\cdot),\hdots,\Psi_{r_C}^{C}(\bm{\theta};\cdot) \big\}.
\end{multline}
We will prove, see \cref{thm:invariant-space-Kh}, that
$\mathcal{F}_h(\bm{\theta})$ is an \emph{invariant} function space for
the operator $K_h$.  

This result is key to LFA, since LFA depends on computing error
reduction factors for different frequencies $\bm{\theta}$. Due to the
invariant function space property for a fixed value of $\bm{\theta}$,
this error reduction factor can be computed for each value of
$\bm{\theta}$ separately. We remark that $\mathcal{F}_h(\bm{\theta})$
is an extension of the invariant space for edge-based operators
introduced in \cite{Boonen2008curl}. We require the following
definitions and notation.

Let
$\widehat{\Psi}^{\alpha}
=[\Psi^{\alpha}_1,\cdots,\Psi^{\alpha}_{r_{\alpha}}]$,
$\alpha \in\{N,X,Y,C\}$ and
\begin{equation}\label{eq:Psi-four-block}
  \boldsymbol{\Psi}
  =
  \begin{bmatrix}
    \widehat{\Psi}^N&   \widehat{\Psi}^X   & \widehat{\Psi}^Y &\widehat{\Psi}^C
  \end{bmatrix},
\end{equation}
and note that any function
$\Phi(\boldsymbol{\theta}; \cdot) \in \mathcal{F}_h(\bm{\theta})$ is a
linear combination of the basis functions
$\Psi^{\alpha}_{j}(\boldsymbol{\theta}; \cdot)$:
\begin{equation}
  \label{eq:r-dimension-function}
  \Phi(\boldsymbol{\theta};\cdot)
  =
  \sum_{\alpha \in\{N,X,Y,C\}} \sum_{ j=1,\cdots, r_{\alpha}} \xi^{\alpha}_{j} \Psi^{\alpha}_{j}(\boldsymbol{\theta};\cdot),
  \qquad \xi^{\alpha}_{j} \in \mathbb{C}.
\end{equation}
Alternatively, we can write
\begin{equation*}
  \Phi(\boldsymbol{\theta};\cdot)
  = \boldsymbol{\Psi \xi},
  \quad
  \text{where}
  \quad
  \boldsymbol{\xi}
  =
  \begin{bmatrix}
    \xi^{N}\\
    \xi^{X}\\
    \xi^{Y}\\
    \xi^{C}
 \end{bmatrix},
 \quad
 \xi^{\alpha}
 =
 \begin{bmatrix}
   \xi^{\alpha}_1\\
   \xi^{\alpha}_2\\
   \vdots \\
   \xi^{\alpha}_{r_{\alpha}}
 \end{bmatrix},
 \quad
 \alpha \in\{N,X,Y,C\}.
\end{equation*}

We now show that $\mathcal{F}_h(\bm{\theta})$ is an invariant function
space for the operator $K_h$.

\begin{theorem}
  \label{thm:invariant-space-Kh}
  For any $\boldsymbol{\theta}$ in $\mathbb{R}^2$ and
  $\Phi(\boldsymbol{\theta}; \cdot) \in
  \mathcal{F}_h(\boldsymbol{\theta})$ it holds that
  \begin{equation}
    \label{eq:invariant-property}
    K_h \Phi(\boldsymbol{\theta}; \cdot) =K_h \boldsymbol{\Psi\xi}
    =\boldsymbol{\Psi} \widetilde{K}_h \boldsymbol{\xi} \in \mathcal{F}_h(\boldsymbol{\theta}),
  \end{equation}
  where the $r\times r$ complex matrix
  $$
  \widetilde{K}_h = (\widetilde{K}_{\alpha_i\beta_j})
  $$ is called the symbol of the operator $K_h$.
\end{theorem}
\begin{proof}
  By
  \cref{eq:r-dimension-function},
  \begin{equation}\label{eq:Kh-acts-alphai-betaj}
    \begin{split}
      K_h\Phi(\boldsymbol{\theta}; \boldsymbol{x})
      &=
      K_h \del[2]{ \sum_{\beta \in\{N,X,Y,C\}}\sum_{j=1,\cdots, r_{\beta}} \xi^{\beta}_{j}
        \Psi^{\beta}_{j}(\boldsymbol{\theta}; \boldsymbol{x}) }
      \\
      &=
      \sum_{\beta \in\{N,X,Y,C\}}\sum_{ j=1,\cdots, r_{\beta}}\xi^{\beta}_{j}K_h
      \Psi^{\beta}_{j}(\boldsymbol{\theta}; \boldsymbol{x}).
    \end{split}
  \end{equation}
  Considering the restriction of the grid function
  $K_h\Phi(\boldsymbol{\theta}; \boldsymbol{x})$ to the $\alpha_i$
  grid, i.e., evaluating
  $K_h\Phi(\boldsymbol{\theta}; \boldsymbol{x})$ in
  $\boldsymbol{x}^{\alpha_i} \in \boldsymbol{G}_h^{\alpha_i}$, we
  obtain, using \cref{eq:Khwxgalphaia}, \cref{eq:Khwxgalphai} and
  \cref{eq:def-K-act-on-Fourier-mode},
  \begin{equation*}
    \begin{split}
      K_h \Psi^{\beta}_{j}(\boldsymbol{\theta}; \boldsymbol{x}^{\alpha_i})
      &=
      \displaystyle\sum_{\gamma \in\{N,X,Y,C\}} \sum_{k=1}^{r_{\gamma}}\big( K_{\alpha_i\gamma_k}\Psi^{\beta}_{j}(\boldsymbol{\theta}; \boldsymbol{x}^{\alpha_i})\big)
      \\
      &=
      \displaystyle\sum_{\gamma \in\{N,X,Y,C\}} \sum_{k=1}^{r_{\gamma}}\big(\sum_{\boldsymbol{\kappa} \in \boldsymbol{V}_{\alpha_i\beta_j}}s_{\kappa}^{(\alpha_i,\gamma_k)}
      \Psi^{\beta}_{j}(\boldsymbol{\theta}; \boldsymbol{x}^{\alpha_i}+\boldsymbol{\kappa} h)\big)
      \\
      &=
      \sum_{\boldsymbol{\kappa}}s_{\boldsymbol{\kappa}}^{(\alpha_i,\beta_j)}\varphi_h(\boldsymbol{\theta};\boldsymbol{x}^{\alpha_i}+\boldsymbol{\kappa}h)
      \\
      &= \sum_{\boldsymbol{\kappa} \in \boldsymbol{V}_{\alpha_i\beta_j}}s^{(\alpha_i,\beta_j)}_{\boldsymbol{\kappa}} e^{\iota \boldsymbol{\theta}\cdot\boldsymbol{\kappa}}
      \varphi_h(\boldsymbol{\theta}; \boldsymbol{x}^{\alpha_i})
      \\
      &=\widetilde{K}_{\alpha_i\beta_j}\varphi_h(\boldsymbol{\theta};\boldsymbol{x}^{\alpha_i})
      \\
      &=\widetilde{K}_{\alpha_i\beta_j}\Psi^{\alpha}_{i}(\boldsymbol{\theta}; \boldsymbol{x}^{\alpha_i}).
    \end{split}
  \end{equation*}
  It then follows from \cref{eq:Kh-acts-alphai-betaj}, on grid
  $\alpha_i$, that
  \begin{equation}\label{eq:K-acts-on-alphai-point}
    \begin{split}
      K_h\Phi(\boldsymbol{\theta}; \boldsymbol{x}^{\alpha_i})
      &=
      \sum_{\beta \in\{N,X,Y,C\}} \sum_{j=1,\cdots,r_{\beta}} \xi^{\beta}_{j}
      \widetilde{K}_{\alpha_i\beta_j}\Psi^{\alpha}_{i}(\boldsymbol{\theta}; \boldsymbol{x}^{\alpha_i}).
    \end{split}
  \end{equation}
  Since $\alpha\in\{N,X,Y,C\}$ and $i\in\{1,2,\cdots,r_{\alpha}\}$ are
  arbitrary, \cref{eq:K-acts-on-alphai-point} implies
  \cref{eq:invariant-property} as follows.  Considering
  $K_h\Phi(\boldsymbol{\theta}; \boldsymbol{x})$ for a generic
  $\boldsymbol{x}\in \widetilde{\boldsymbol{G}}_h$ on any of the $r$
  subgrids, we can sum the right hand side of
  \cref{eq:K-acts-on-alphai-point} over all grids to obtain
  \begin{equation*}
    \begin{split}
      K_h\Phi(\boldsymbol{\theta}; \boldsymbol{x})
      &=
      \sum_{\alpha \in\{N,X,Y,C\}} \sum_{ i=1,\cdots, r_{\alpha}} \Psi^{\alpha}_{i}(\boldsymbol{\theta};\boldsymbol{x}) \sum_{\beta \in\{N,X,Y,C\}} \sum_{j=1,\cdots,r_{\beta}} \xi^{\beta}_{j}
      \widetilde{K}_{\alpha_i\beta_j}
      \\
      &= \boldsymbol{\Psi} \widetilde{K}_h \boldsymbol{\xi},
    \end{split}
  \end{equation*}
  using the fact that the
  $\Psi^{\alpha}_{i}(\boldsymbol{\theta};\boldsymbol{x})$ have zero
  overlap in the sum due to \cref{eq:-four-basis}.
\end{proof}

We now consider frequency aliasing and generalize \cite[Theorem
3.2]{Boonen2008curl}.

\begin{theorem}
  \label{thm:2pi-equal-property}
  Let
  $\boldsymbol{\eta} = (\eta_1,\eta_2) \in
  \cbr[0]{(0,0),(1,0),(0,1),(1,1)}$. For any
  $\boldsymbol{\theta} \in \mathbb{R}^2$,
  $\boldsymbol{x} \in \widetilde{\boldsymbol{G}}_h$, and
  $\Phi(\boldsymbol{\theta};\boldsymbol{x})$ in
  $\mathcal{F}_h(\boldsymbol{\theta})$, we have that
  \begin{equation}
    \label{distiguish-Lh-2pi}
    \Phi(\boldsymbol{\theta}+2\pi\boldsymbol{\eta}; \boldsymbol{x})
    =
    \Phi(\boldsymbol{\theta}; \boldsymbol{x}) \times
    \begin{cases}
      1
      & \text{if } \boldsymbol{x} \in \bigotimes_{i\in\{1,\cdots,r_N\}}\boldsymbol{G}^{N_i}_h,
      \\
      (-1)^{\eta_1}
      & \text{if } \boldsymbol{x}\in \bigotimes_{i\in\{1,\cdots,r_X\}}\boldsymbol{G}^{X_i}_h,
      \\
      (-1)^{\eta_2}
      & \text{if } \boldsymbol{x}\in \bigotimes_{i\in\{1,\cdots,r_Y\}}\boldsymbol{G}^{Y_i}_h,
      \\
      (-1)^{\eta_1+\eta_2}
      & \text{if } \boldsymbol{x}\in \bigotimes_{i\in\{1,\cdots,r_C\}}\boldsymbol{G}^{C_i}_h.
    \end{cases}
  \end{equation}
\end{theorem}
\begin{proof}
  By \cref{eq:r-dimension-function} and \cref{eq:-four-basis} we need
  to verify that
  \begin{subequations}
    \begin{align}
      \label{eq:varphihGNi_a}
      &\varphi_h(\boldsymbol{\theta}+2\pi{\boldsymbol{\eta}}; \boldsymbol{x})
        =\varphi_h(\boldsymbol{\theta};\boldsymbol{x}),
      && \text{if}\quad \boldsymbol{x} \in \bigotimes_{i=1,\cdots,r_N}\boldsymbol{G}_h^{N_i},
      \\
      \label{eq:varphihGNi_b}
      &\varphi_h(\boldsymbol{\theta}+2\pi{\boldsymbol{\eta}};\boldsymbol{x})
        =(-1)^{\eta_1}\varphi_h(\boldsymbol{\theta};\boldsymbol{x}),
      && \text{if}\quad \boldsymbol{x} \in \bigotimes_{i=1,\cdots,r_X}\boldsymbol{G}_h^{X_i},
      \\
      \label{eq:varphihGNi_c}
      &\varphi_h(\boldsymbol{\theta}+2\pi{\boldsymbol{\eta}};\boldsymbol{x})
        =(-1)^{\eta_2}\varphi_h(\boldsymbol{\theta};\boldsymbol{x}),
      && \text{if}\quad \boldsymbol{x} \in \bigotimes_{i=1,\cdots,r_Y}\boldsymbol{G}_h^{Y_i},
      \\
      \label{eq:varphihGNi_d}
      &\varphi_h(\boldsymbol{\theta}+2\pi{\boldsymbol{\eta}};\boldsymbol{x})
        =(-1)^{\eta_1+\eta_2}\varphi_h(\boldsymbol{\theta};\boldsymbol{x}),
      && \text{if}\quad \boldsymbol{x} \in \bigotimes_{i=1,\cdots,r_C}\boldsymbol{G}_h^{C_i}.
    \end{align}
  \end{subequations}
  We show \cref{eq:varphihGNi_a}. If
  $\boldsymbol{x} \in \boldsymbol{G}_h^{N}$ then
  $\boldsymbol{x} = (ih, jh)$ with $i,j\in\mathbb{Z}$. Then
  \begin{equation*}
    \varphi_h(\boldsymbol{\theta} + 2\pi\boldsymbol{\eta}; \boldsymbol{x})
    =
    e^{\iota (\theta_1 + 2\pi\eta_1)i} e^{\iota (\theta_2 + 2\pi\eta_2) j}
    =
    \varphi_h(\boldsymbol{\theta}; \boldsymbol{x})
    e^{\iota 2\pi\eta_1 i} e^{\iota 2\pi\eta_2 j}
    =
    \varphi_h(\boldsymbol{\theta}; \boldsymbol{x}).
  \end{equation*}
  \Cref{eq:varphihGNi_b,eq:varphihGNi_c,eq:varphihGNi_d} follow
  similarly.
\end{proof}

Due to the frequency aliasing as shown by
\cref{thm:2pi-equal-property}, it is sufficient to consider
$\boldsymbol{\theta} = (\theta_1,\theta_2) \in [-\pi/2, 3\pi/2)^2$ (or
any pair of intervals with length $2\pi$) in the LFA analysis of our
multigrid methods.

%------------------------------------------------------------------------------
\subsection{Two-grid LFA}

We now determine the symbol of the two-grid error propagation operator
\cref{TG-error-form}. In order to apply LFA to the two-grid error
propagation operator we need to determine how the operators $K_h$,
$P_H^h$, $R_h^H$, $S_h$, and $K_H$ act on the Fourier components
$\Phi(\boldsymbol{\theta};\cdot)$ in
$\mathcal{F}_h(\boldsymbol{\theta})$. Note that the Fourier modes on
the coarse grid $\boldsymbol{\widetilde{G}}_H$ are given by
\begin{equation}
  \label{eq:r-dimension-function-H}
  \Phi_H(\boldsymbol{\theta};\cdot)
  =
  \sum_{\alpha \in\{N,X,Y,C\}} \sum_{ j=1,\cdots, r_{\alpha}} \xi^{\alpha}_{j}
  \bar{\Psi}^{\alpha}_{j}(\boldsymbol{\theta};\cdot),
\end{equation}
where
\begin{equation}
  \label{eq:-four-basis-H}
  \bar{\Psi}^{\alpha}_{j}(\boldsymbol{\theta}; \boldsymbol{x})
  =
  \begin{cases}
    \varphi_H(\boldsymbol{\theta}; \boldsymbol{x})
    & \text{if } \boldsymbol{x}\in\boldsymbol{G}^{\alpha_j}_H,
    \\
    0
    & \text{otherwise}.
  \end{cases}
\end{equation}

We have the following properties of Fourier modes on
$\boldsymbol{\widehat{G}}_H$.

\begin{theorem}
  \label{thm:alising-property}
  Let
  $\boldsymbol{\eta} = (\eta_1,\eta_2) \in
  \cbr[0]{(0,0),(1,0),(0,1),(1,1)}$.  For any
  $\boldsymbol{\theta} \in \mathbb{R}^2$,
  $\boldsymbol{x} \in \boldsymbol{\widehat{G}}_H$, and
  $\Phi(\boldsymbol{\theta};\boldsymbol{x})$ in
  $\mathcal{F}_h(\boldsymbol{\theta})$, we have that
  \begin{equation}
    \label{distiguish-LH}
    \Phi(\boldsymbol{\theta} + \pi\boldsymbol{\eta}; \boldsymbol{x})
    =
    \Phi(\boldsymbol{\theta}; \boldsymbol{x}) \times
    \begin{cases}
      1
      & \text{if } \boldsymbol{x}\in \bigotimes_{i\in\{1,\cdots,r_N\}}\boldsymbol{G}^{N_i}_H,
      \\
      (-1)^{\eta_1}
      & \text{if } \boldsymbol{x}\in \bigotimes_{i\in\{1,\cdots,r_X\}}\boldsymbol{G}^{X_i}_H,
      \\
      (-1)^{\eta_2}
      & \text{if } \boldsymbol{x}\in \bigotimes_{i\in\{1,\cdots,r_Y\}}\boldsymbol{G}^{Y_i}_H,
      \\
      (-1)^{\eta_1+\eta_2}
      & \text{if } \boldsymbol{x}\in \bigotimes_{i\in\{1,\cdots,r_C\}}\boldsymbol{G}^{C_i}_H.
    \end{cases}
  \end{equation}

  Furthermore, for $\boldsymbol{x} \in\widehat{\boldsymbol{G}}_H$ and
  $\boldsymbol{\theta}\in[-\pi/2, \pi/2)^2$ we have
  \begin{equation}
    \label{coinside-h-H-modes}
    \Phi(\boldsymbol{\theta} + \pi\boldsymbol{\eta}; \boldsymbol{x})
    =
    \Phi_{H}(2\boldsymbol{\theta}; \boldsymbol{x}) \times
    \begin{cases}
      1
      & \text{if } \boldsymbol{x}\in \bigotimes_{i\in\{1,\cdots,r_N\}}\boldsymbol{G}^{N_i}_H,
      \\
      (-1)^{\eta_1}
      & \text{if } \boldsymbol{x}\in \bigotimes_{i\in\{1,\cdots,r_X\}}\boldsymbol{G}^{X_i}_H,
      \\
      (-1)^{\eta_2}
      & \text{if } \boldsymbol{x}\in \bigotimes_{i\in\{1,\cdots,r_Y\}}\boldsymbol{G}^{Y_i}_H,
      \\
      (-1)^{\eta_1+\eta_2}
      & \text{if } \boldsymbol{x}\in \bigotimes_{i\in\{1,\cdots,r_C\}}\boldsymbol{G}^{C_i}_H.
    \end{cases}
  \end{equation}
\end{theorem}
\begin{proof}
  We omit the proof since it is similar to the proof of
  \cref{thm:2pi-equal-property}.
\end{proof}

We denote the high and low frequency intervals as
\begin{equation*}
  T^{\text{low}}
  = [-\tfrac{\pi}{2}, \tfrac{\pi}{2})^{2},
  \qquad
  T^{\text{high}}
  = [-\tfrac{\pi}{2}, \tfrac{3\pi}{2})^{2}
  \backslash
  [-\tfrac{\pi}{2},\tfrac{\pi}{2})^{2}.
\end{equation*}
Aliasing of modes will occur in the intergrid transfer operations (see
\cite{Trottenberg:book, Wienands:2005}). For this reason we introduce
the $4r$-dimensional harmonic space
$\mathfrak{F}_h(\boldsymbol{\theta})$, with
$\boldsymbol{\theta}\in T^{\rm low}$, as
\begin{equation*}
  \mathfrak{F}_h(\boldsymbol{\theta})
  =\text{span}\cbr[1]{\boldsymbol{\psi}(\boldsymbol{\theta^{(\eta)}} ; \cdot) :
    \boldsymbol{\psi} \in \mathcal{F}_h(\boldsymbol{\theta}), \,
    \boldsymbol{\eta} = (\eta_1, \eta_2) \in \cbr[0]{(0,0),(1,0),(0,1),(1,1)} },
\end{equation*}
where
$\boldsymbol{\theta}^{(\boldsymbol{\eta})}=\boldsymbol{\theta}^{(\eta_1,\eta_2)}:=(\theta_1+\pi
\eta_1,\theta_2+\pi \eta_2) = \boldsymbol{\theta} + \pi
\boldsymbol{\eta}$ and
$\boldsymbol{\theta} = \boldsymbol{\theta}^{(0,0)} \in T^{{\rm
    low}}$. Note that every function
$\boldsymbol{\psi} \in \mathfrak{F}_h(\boldsymbol{\theta})$ can be
written as follows, using the $\boldsymbol{\Psi}$ matrix defined in
\cref{eq:Psi-four-block}:
\begin{equation}
  \label{eq:-two-grid-basis-expansion}
  \boldsymbol{\psi}
  = \boldsymbol{\Psi}(\boldsymbol{\theta}^{(0,0)};\cdot)\,\boldsymbol{\xi}^{(0,0)}
  + \boldsymbol{\Psi}(\boldsymbol{\theta}^{(1,0)};\cdot)\,\boldsymbol{\xi}^{(1,0)}
  + \boldsymbol{\Psi}(\boldsymbol{\theta}^{(0,1)};\cdot)\,\boldsymbol{\xi}^{(0,1)}
  + \boldsymbol{\Psi}(\boldsymbol{\theta}^{(1,1)};\cdot)\,\boldsymbol{\xi}^{(1,1)},
\end{equation}
where the $\boldsymbol{\xi}^{\boldsymbol{\eta}}\in\mathbb{C}^{r\times 1}$
are uniquely determined.

We remark that $\mathfrak{F}_h(\boldsymbol \theta)$ is invariant for
the two-grid error operator (we refer to \cite[Section
4.4]{Trottenberg:book} for a discussion on how to prove this for
node-based problems), i.e., for any $\boldsymbol{\psi}$ in
\cref{eq:-two-grid-basis-expansion},
\begin{equation*}
  \begin{split}
    E_h \boldsymbol{\psi}
    &= E_h \sbr[1]{\boldsymbol{\Psi}(\boldsymbol{\theta}^{(0,0)};\cdot) \
      \boldsymbol{\Psi}(\boldsymbol{\theta}^{(1,0)};\cdot) \
      \boldsymbol{\Psi}(\boldsymbol{\theta}^{(0,1)};\cdot) \
      \boldsymbol{\Psi}(\bm{\theta}^{(1,1)};\cdot) } \widehat{\boldsymbol{\xi}}
    \\
    &= \sbr[1]{\boldsymbol{\Psi}(\boldsymbol{\theta}^{(0,0)};\cdot) \
      \boldsymbol{\Psi}(\boldsymbol{\theta}^{(1,0)};\cdot) \
      \boldsymbol{\Psi}(\boldsymbol{\theta}^{(0,1)};\cdot) \
      \boldsymbol{\Psi}(\boldsymbol{\theta}^{(1,1)};\cdot)} \widetilde{\boldsymbol{E}}_h \boldsymbol{\widehat{\xi}},
  \end{split}
\end{equation*}
where
$\widehat{\boldsymbol{\xi}}=\sbr[0]{\boldsymbol{\xi}^{(0,0)} \
  \boldsymbol{\xi}^{(1,0)} \ \boldsymbol{\xi}^{(0,1)} \
  \boldsymbol{\xi}^{(0,0)}}^T \in \mathbb{C}^{4r\times 1}$, and
$\widetilde{\boldsymbol{E}}_h$ is the $4r \times 4r$ matrix which is
the LFA representation of the two-grid operator $E_h$ given by
\begin{equation}
  \label{eq:symbolEh}
  \widetilde{ \boldsymbol{E}}_h(\boldsymbol{\theta},\omega)
  = \widetilde{\boldsymbol {S}}^{\nu_2}_h(\boldsymbol{\theta},\omega)
  \del{I-\widetilde{\boldsymbol {P}}_H^h(\boldsymbol{\theta})
    (\widetilde{K}_{H}(2\boldsymbol{ \theta}))^{-1}
    \widetilde{\boldsymbol{ R}}_h^H(\boldsymbol{\theta})
    \widetilde{\boldsymbol{ \mathcal{K}}}_{h}(\boldsymbol{\theta})}
  \widetilde{\boldsymbol{ S}}^{\nu_1}_h(\boldsymbol{\theta},\omega),
\end{equation}
where
\begin{eqnarray*}
  \widetilde{\boldsymbol{\mathcal{K}}}_h(\boldsymbol{\theta})
  &=&\text{diag}\cbr[1]{\widetilde{K}_h(\boldsymbol{\theta}^{(0,0)}),
      \widetilde{K}_h(\boldsymbol{\theta}^{(1,0)}),\widetilde{K}_h(\boldsymbol{\theta}^{(0,1)}),
      \widetilde{K}_h(\boldsymbol{\theta}^{(1,1)}) },
  \\
  \widetilde{\boldsymbol{S}}_h(\boldsymbol{\theta},\omega)
  &=&\text{diag}\cbr[1]{\widetilde{S}_h(\boldsymbol{\theta}^{(0,0)},\omega),
      \widetilde{S}_h(\boldsymbol{\theta}^{(1,0)},\omega),\widetilde{S}_h(\boldsymbol{\theta}^{(0,1)},\omega),
      \widetilde{S}_h(\boldsymbol{\theta}^{(1,1)},\omega) },
  \\
  \widetilde{\boldsymbol{R}}_h^H(\boldsymbol{\theta})
  &=&\del[1]{ \widetilde{R}_h^H(\boldsymbol{\theta}^{(0,0)}), \widetilde{R}_h^H(\boldsymbol{\theta}^{(1,0)}),
      \widetilde{R}_h^H(\boldsymbol{\theta}^{(0,1)}), \widetilde{R}_h^H(\boldsymbol{\theta}^{(1,1)}) },
  \\
  \widetilde{\boldsymbol{P}}_H^h(\boldsymbol{\theta})
  &=&\del[1]{ \widetilde{P}_H^h(\boldsymbol{\theta}^{(0,0)}); \widetilde{P}_H^h(\boldsymbol{\theta}^{(1,0)});
      \widetilde{P}_H^h(\boldsymbol{\theta}^{(0,1)});\widetilde{P}_H^h(\boldsymbol{\theta}^{(1,1)}) },
\end{eqnarray*}
in which $\text{diag}\cbr[0]{A_1,A_2,A_3,A_4}$ refers to a block
diagonal matrix with diagonal blocks $A_1$, $A_2$, $A_3$, and
$A_4$. Furthermore, $\widetilde{K}_h$ is the symbol of the operator
$K_h$ as discussed in \cref{sub:invariant-space} and $\widetilde{S}_h$
is the symbol of the additive Vanka-type smoother (we refer to
\cite[Sections 3.5 and 3.6]{farrell2019local} on how to compute this
symbol). We refer to \cite[Section 5.2]{hetwo}, \cite[Section
3.4]{farrell2019local}, and \cite{Maclachlan:2011} on how to
compute $\widetilde{\boldsymbol{R}}_h^H$, the symbol of $R_h^H$,
taking into account the different groups of DOFs. The symbol of
$P_H^h$, $\widetilde{\boldsymbol{P}}_H^h$ follows from the symbol of
$R_h^H$ since $R_h^H = (P_H^h)^T$. Similarly, the symbol of the
coarse-grid operator, $\widetilde{K}_H$, follows from the symbols of
the grid-transfer operators and the fine-grid operator.

Given the symbol of $E_h$ \cref{eq:symbolEh}, the two-grid LFA
asymptotic convergence factor, $\rho_{{\rm asp}}$, is defined as
\cite{Trottenberg:book, Wienands:2005} 
\begin{equation}
  \label{real-TGM}
  \rho_{\text{asp}} =
  \sup_{ \boldsymbol{\theta} \in T^{\text{low}}}
  \cbr[1]{ \rho \del[0]{ \widetilde{\boldsymbol{E}}_h(\boldsymbol{\theta}, \omega) } },
\end{equation}
where $\rho(\widetilde{\boldsymbol{E}}_h(\boldsymbol{\theta},\omega))$
is the spectral radius of the matrix
$\widetilde{\boldsymbol{E}}_h$. In \cref{sec:numer} we will
approximate $\rho_{\text{asp}}$ by sampling over a finite set of
frequencies. In general, this approximation provides a sharp
prediction of the two-grid performance.

%------------------------------------------------------------------------------
\section{Numerical Results}
\label{sec:numer}

\begin{figure}[tbp]
  \centering
  \subfloat[HDG, vertex-wise Vanka. \label{fig:HDG-Vertex-CS}]
  {\includegraphics[width=0.4\textwidth]{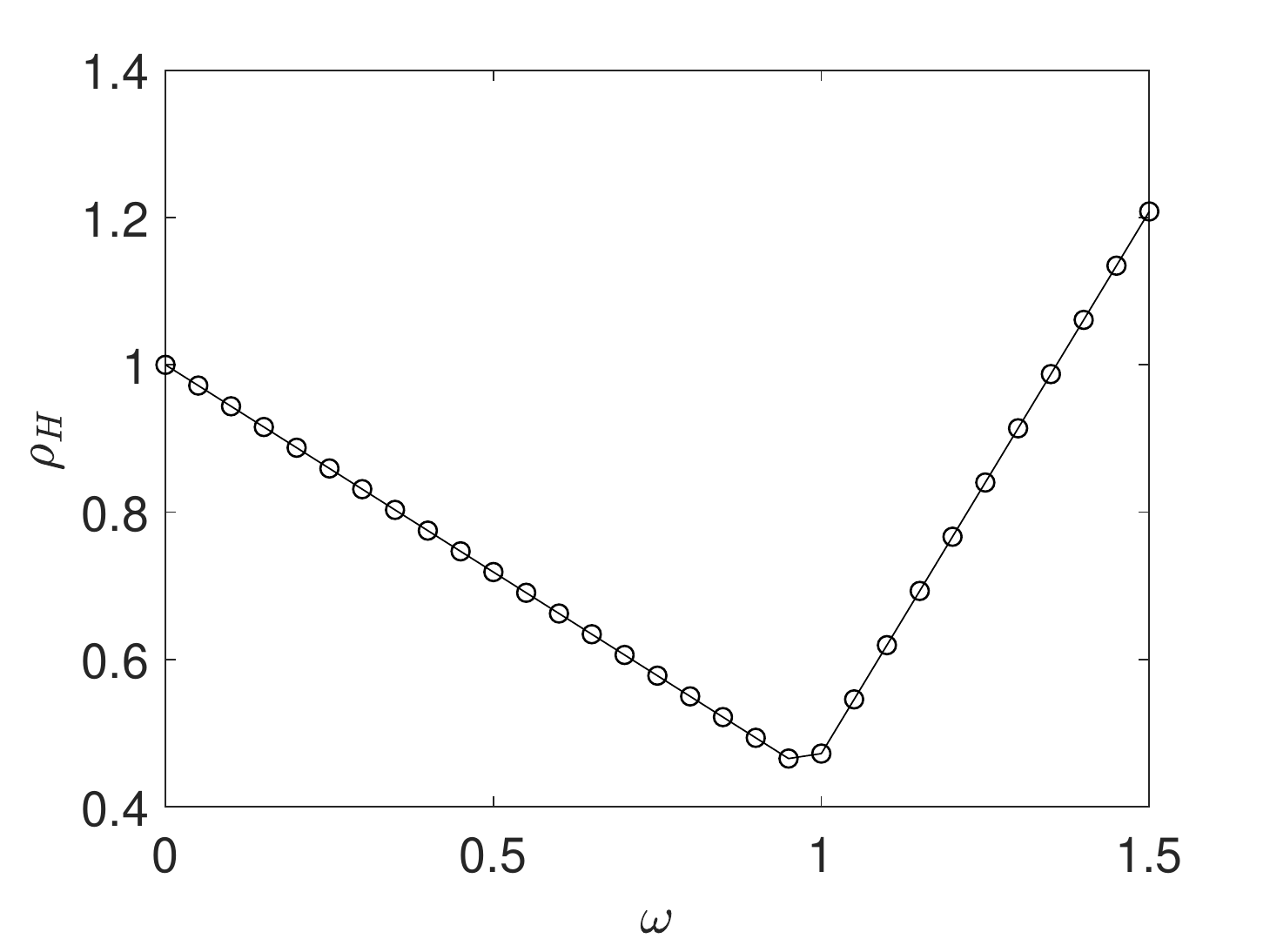}}
  \subfloat[HDG, element-wise Vanka. \label{fig:HDG-Element-CS}]
  {\includegraphics[width=0.4\textwidth]{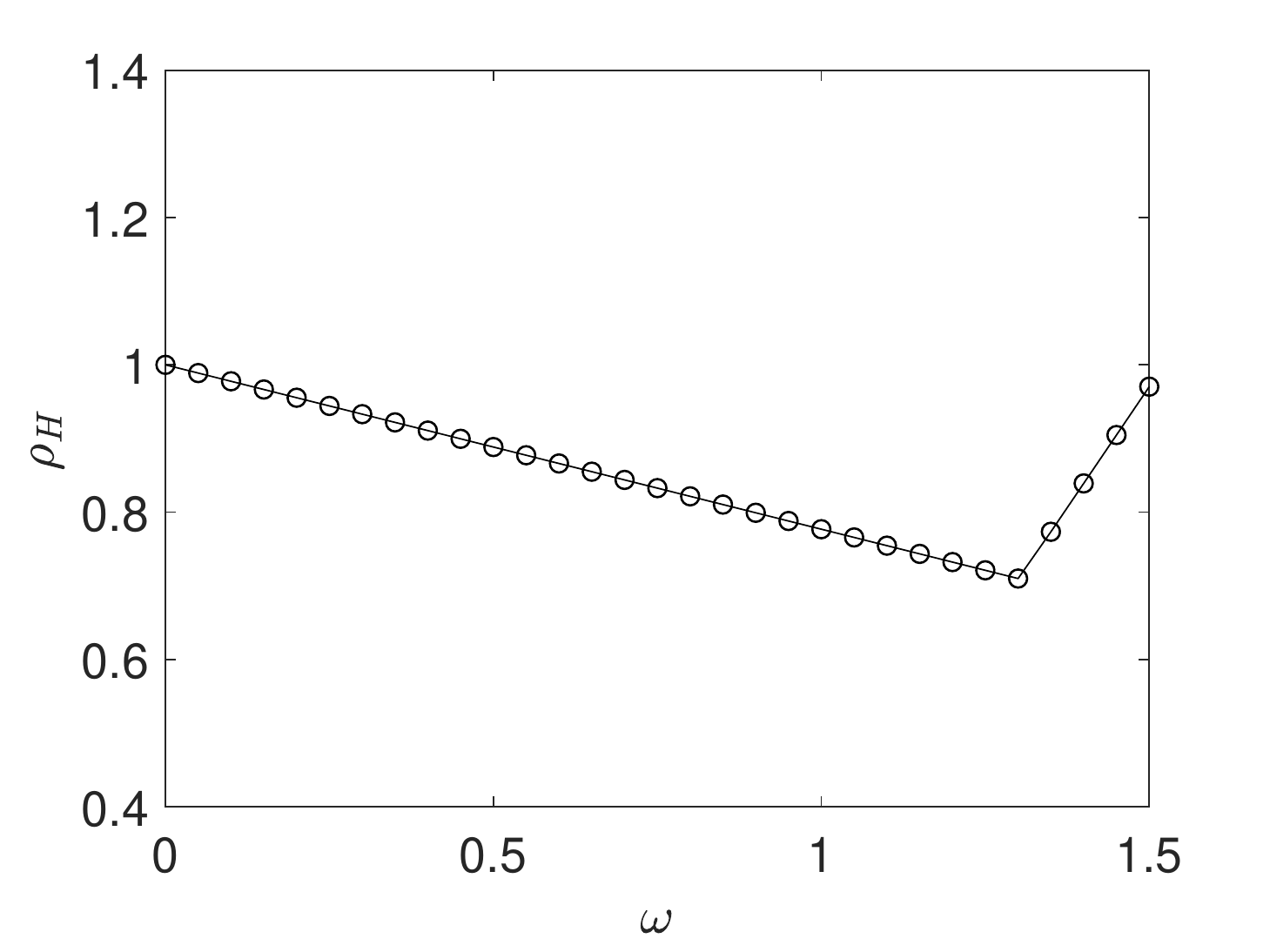}}
  \\
  \subfloat[EDG, vertex-wise Vanka. \label{fig:EDG-Vertex-CS}]
  {\includegraphics[width=0.4\textwidth]{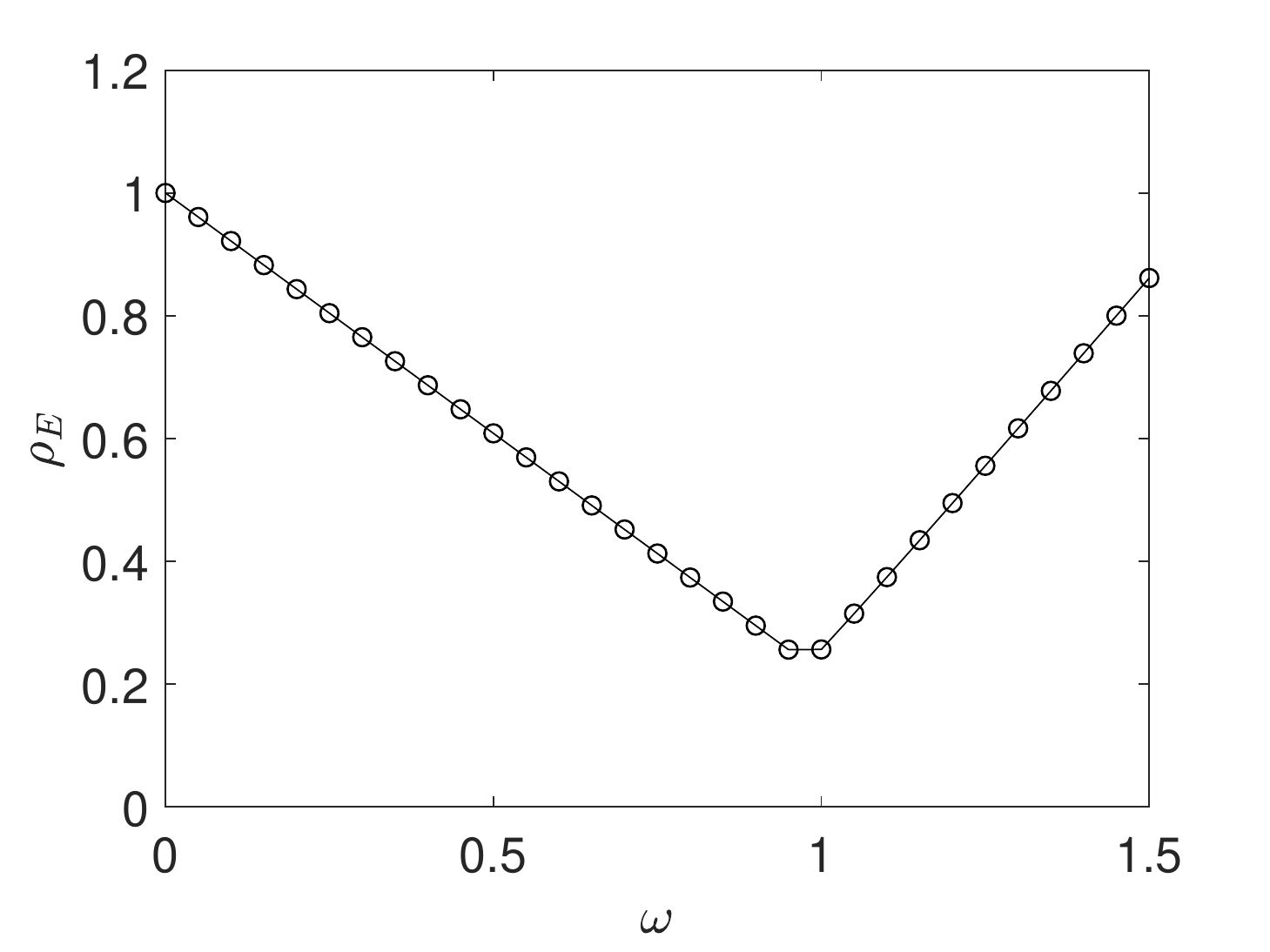}}
  \subfloat[EDG, element-wise Vanka. \label{fig:EDG-Element-CS}]
  {\includegraphics[width=0.4\textwidth]{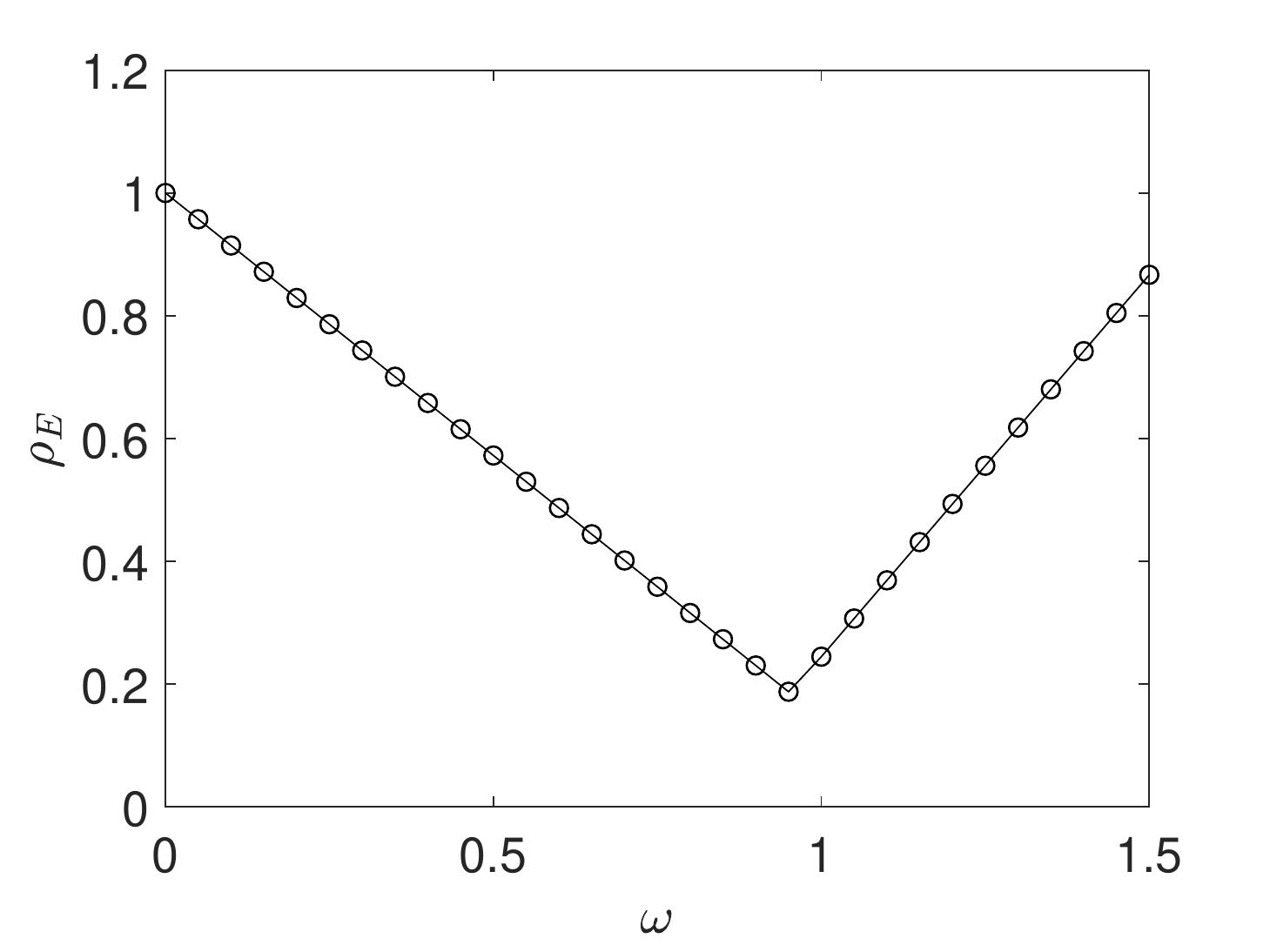}}
  \caption{Variation of two-grid LFA convergence factor as a function
    of $\omega$. Left column: the two-grid method using vertex-wise
    additive Vanka relaxation schemes. Right column: the two-grid
    method using element-wise additive Vanka relaxation schemes. The
    two-grid method is applied to HDG and EDG discretizations of the
    Laplacian with $k=2$.}
  \label{fig:EDG-HDG-k2-Element-CS-vs-parameter}
\end{figure}

We will now use the LFA from \cref{sec:LFA-tg} to compare the
efficiency of geometric multigrid (see \cref{sec:two-grid}) for
solving linear systems resulting from HDG, EDG, and CG discretizations
of the Poisson problem. We furthermore compare the performance of the
additive Vanka smoothers introduced in \cref{subsec:relaxation-scheme}
to the classical relaxation iterations of weighted pointwise Jacobi
and Gauss--Seidel. For a fair comparison we use the ordering of the
DOFs defined by the Vanka patches for all smoothers. In particular,
the DOFs are ordered with $N$-type DOFs first, followed by $X$-type
DOFs, $Y$-type DOFs, and $C$-type DOFs. Lexicographical ordering is
used for each group of DOFs. In all our simulations we consider meshes
with square elements and linear, quadratic, and cubic ($k=1,2,3$)
polynomial approximations. In the HDG and EDG discretization
\cref{eq:edghdg} we furthermore set the penalty parameter to
$\alpha = 6 k^2$ \cite{Riviere:book}.

%------------------------------------------------------------------------------
\subsection{LFA prediction of multigrid efficiency}
\label{ss:lfaprediction}

\begin{table}[tbp]
  \footnotesize
  \caption{Two-grid LFA predictions with $\nu_1=1$ and $\nu_2=0$ of
    the spectral radii for CG, HDG, and EDG discretizations of the
    Laplacian. The number in brackets is the optimal value of $\omega$
    that minimizes the spectral radius. Here $\rho_C$, $\rho_E$,
    $\rho_H$ are the spectral radii of the two-grid method applied to
    a CG, EDG, and HDG discretization, respectively. We consider
    vertex-wise (VW), Lower-Triangular-Vertex-Wise (LTVW),
    element-wise (EW), and Lower-Triangular-Element-Wise (LTEW)
    additive Vanka type relaxation, pointwise Jacobi (JAC) relaxation
    and Gauss--Seidel (GS) relaxation. For each row, the best
    convergence factor among the smoothers that can be executed in
    parallel (i.e., excluding GS) is indicated in bold.}
  \centering
  \begin{tabular}{l|l|l|l|l|l||l}
    \hline
    $\rho$     &VW  & EW    & JAC    &LTVW   & LTEW   & GS \\
    \hline
    \multicolumn{7} {c} {\footnotesize{$k=1$}} \\
    \hline
    $\rho_C(\rho_E)$        &0.333(0.89)      &{\bf 0.200}(0.90)               &0.333(0.89)    &0.333(0.89) &0.282(0.90)  &0.261(1.02)\\
    $\rho_H$                &{\bf 0.403}(0.96)      &0.466(1.14)               &0.801(0.76)    &0.609(1.12)  &0.604(1.18)  &0.394(1.30)\\
    \hline
    \multicolumn{7} {c} {\footnotesize{$k=2$}} \\
    \hline
    $\rho_C$                 &0.208(1.00)         &0.282(0.84)          &0.452(1.00)          &0.362(1.02)  &{\bf 0.188}(1.02) &0.167(1.06)\\
    $\rho_E$                 &0.233(0.98)         &{\bf 0.194}(0.96)          &0.537(1.02)          &0.423(1.08)  &0.325(1.10) &0.246(1.10) \\
    $\rho_H$                 &{\bf 0.449}(0.98)         &0.710(1.30)           &0.893(0.82)         &0.802(1.18)  &0.799(1.20) &0.628(1.50) \\
    \hline
    \multicolumn{7} {c} {\footnotesize{$k=3$}} \\
    \hline
    $\rho_C$                 &0.233(0.96)        &{\bf 0.203}(0.94)        &0.654(0.78)      &0.368(1.08)   &0.301(1.00)  &0.233(1.10)\\
    $\rho_E$                 &{\bf 0.287}(0.94)        &0.332(1.10)        &0.792(0.90)      &0.598(1.20)   &0.576(1.18)  &0.470(1.30)\\
    $\rho_H$                &{\bf  0.476}(0.98)        &0.794(1.32)         &0.932(0.78)      &0.862(1.22)   &0.862(1.22)  &0.745(1.50)\\
    \hline
  \end{tabular}
  \label{tab:Vanka-single-two-grid-LFA}
\end{table}

\begin{table}[tbp]
  \small
  \caption{Two-grid LFA predictions with different pre- and post
    relaxation sweeps of the spectral radii for CG, HDG, and EDG
    discretizations of the Laplacian. Here $\rho_C$, $\rho_E$,
    $\rho_H$ are the spectral radii of the two-grid method applied to
    a CG, EDG, and HDG discretization, respectively. Furthermore
    $\text{TG}(\nu_1, \nu_2)$ denotes the two-grid method using
    $\nu_1$ and $\nu_2$ pre- and post relaxation sweeps.}
  \centering
  \begin{tabular}{l|l|l|l|l}
    \hline
    $k$ & $\rho$ & TG(1,1) & TG(1,2) & TG(2,2) \\
    \hline
    \multicolumn{5} {c} {} \\
    \multicolumn{5} {c} {Gauss--Seidel relaxation} \\
    \multicolumn{5} {c} {} \\
    \hline
    \hline
    1 & $\rho_C(\rho_E)$  &0.079         &0.069         &0.034   \\
      & $\rho_H$          &0.238         &0.108         &0.058   \\
    \hline
    2 & $\rho_C$          &0.071         &0.042         &0.028   \\
      & $\rho_E$          &0.083         &0.047         &0.025   \\
      & $\rho_H$          &0.387        &0.241         &0.157     \\
    \hline
    3 & $\rho_C$          &0.095         &0.044         &0.023   \\
      & $\rho_E$          &0.221         &0.115         &0.078   \\
      & $\rho_H$          &0.555         &0.410         &0.306   \\
    \hline
    \hline
    \multicolumn{5} {c} {} \\
    \multicolumn{5} {c} {Vertex-Wise Vanka relaxation} \\
    \multicolumn{5} {c} {} \\
    \hline
    \hline
    1 & $\rho_C(\rho_E)$ &0.112        &0.078       &0.061  \\
      & $\rho_H$         &0.250        &0.149       &0.093   \\
    \hline
    2 & $\rho_C$         &0.064                &0.022               &0.012   \\
      & $\rho_E$         &0.096                &0.042               &0.034   \\
      & $\rho_H$         &0.225                &0.105               &0.071     \\
    \hline
    3 & $\rho_C$         &0.068                &0.023              &0.014  \\
      & $\rho_E$         &0.122                &0.057              &0.042   \\
      & $\rho_H$         &0.227                &0.114              &0.073   \\
    \hline
    \hline
    \multicolumn{5} {c} {} \\
    \multicolumn{5} {c} {Element-Wise Vanka relaxation} \\
    \multicolumn{5} {c} {} \\
    \hline
    \hline
    1 & $\rho_C(\rho_E)$ &0.090       &0.048    & 0.040  \\
      & $\rho_H$         &0.342       &0.193    &0.138   \\
    \hline
    2 & $\rho_C$ &0.079        &0.022       &0.009 \\
      & $\rho_E$ &0.094        &0.045       &0.032   \\
      & $\rho_H$ &0.518        &0.360       &0.274     \\
    \hline
    3 & $\rho_C$  &0.104       &0.032    &0.017 \\
      & $\rho_E$  &0.165       &0.063    &0.043   \\
      & $\rho_H$  &0.631       &0.501    &0.398   \\
    \hline
    \hline
    \multicolumn{5} {c} {} \\
    \multicolumn{5} {c} {Lower-Triangular-Element-Wise Vanka relaxation} \\
    \multicolumn{5} {c} {} \\
    \hline
    \hline
    1 & $\rho_C(\rho_E)$  &0.098         &0.070         &0.052   \\
      & $\rho_H$          &0.456         &0.284         &0.226   \\
    \hline
    2 & $\rho_C$          &0.065         &0.043         &0.030   \\
      & $\rho_E$          &0.180         &0.076         &0.050   \\
      & $\rho_H$          &0.672         &0.529         &0.456     \\
    \hline
    3 & $\rho_C$          &0.117         &0.057         &0.032   \\
      & $\rho_E$          &0.327         &0.185         &0.119   \\
      & $\rho_H$          &0.743         &0.640         &0.551   \\
    \hline
    \hline
  \end{tabular}
  \label{tab:testing_sweeps}
\end{table}

The additive Vanka \cref{eq:vanka-omega}, weighted pointwise Jacobi
and Gauss--Seidel relaxation schemes include a tunable parameter
$\omega$. By a brute-force approach we first determine the optimal
value of $\omega$ that minimizes the two-grid convergence factor
$\rho_{\text{asp}}$ (see \cref{real-TGM}) for each discretization and
for each polynomial degree. For this we use $32 \times 32$ evenly
distributed Fourier frequencies $\boldsymbol{\theta}$ in the Fourier
domain $[-\frac{\pi}{2}+\epsilon, \frac{\pi}{2}-\epsilon]^2$ with
$\epsilon =\pi/64$.

We first consider the case where we only use one pre-relaxation sweep
and no post-relaxation sweeps in the two-grid method, i.e., $\nu_1=1$
and $\nu_2 = 0$. In \cref{fig:EDG-HDG-k2-Element-CS-vs-parameter} we
first plot the sensitivity of the two-grid LFA convergence factor with
element- and vertex-wise additive Vanka relaxation schemes for the EDG
and HDG discretization methods when $k=2$. These plots indicate the
importance of correctly choosing $\omega$.

Next we present in \cref{tab:Vanka-single-two-grid-LFA} the computed
optimal value of $\omega$ and the corresponding two-grid LFA
convergence factor for the different discretization and relaxation
methods. For the continuous Galerkin method with $k=1$ and a pointwise
Jacobi relaxation method (which for $CG$ with $k=1$ is the same as a
vertex-wise Vanka patch) it was proven that $\omega = 0.89$ is the
optimal damping parameter \cite[Lemma 4.1]{He:2019}. We therefore take
$\omega = 0.89$ for this particular case and compute only the
corresponding two-grid LFA convergence factor.

As expected, for $k=1$, the spectral radius of the CG and EDG method
are equal \cite{Cockburn:2009b}. For $k=2$ and $k=3$ the LFA predicts
that the performance of the two-grid method for EDG and CG are
similar. We also note from \cref{tab:Vanka-single-two-grid-LFA} that
the two-grid methods for EDG and CG outperform the two-grid method
applied to an HDG discretization, no matter which relaxation scheme is
used.

There is little difference between using vertex-wise and element-wise
Vanka patches and Gauss--Seidel for the CG and EDG
discretizations. For the HDG discretization the vertex-wise Vanka
patch outperforms the element-wise Vanka patch. For all discretization
methods LFA predicts that both the vertex-wise and
element-wise Vanka patch approaches 
lead to smaller convergence factors than pointwise Jacobi.

We furthermore observe that although using a lower-triangular
approximation to $K_i$ in the additive Vanka relaxation scheme does
not affect the two-grid LFA convergence factor too much for the CG and
EDG discretization, the performance of the two-grid method for HDG
significantly deteriorates with this approximation. We finally remark
that, especially for the pointwise Jacobi smoother, the performance of
the two-grid method slowly deteriorates with increasing $k$.

We next consider the performance of the two-grid method using
different pre- and post relaxation sweeps. For $\omega$ we use the
values from \cref{tab:Vanka-single-two-grid-LFA}. In
\cref{tab:testing_sweeps} we present the results predicted by LFA for
the two-grid method using vertex-wise, element-wise and
Lower-Triangular-Element-Wise additive Vanka type relaxation. We do
not present the results here for pointwise Jacobi and
Lower-Triangular-Vertex-Wise Vanka type relaxation because, for the
same pre- and post relaxation sweeps, these are always worse than the
previously mentioned relaxation schemes.

For all discretization and relaxation methods we observe that the
convergence factor of the two-grid method is reduced when the number
of pre- and post-relaxation sweeps are increased, however, this
improvement is most notable for the CG and EDG discretizations. We
furthermore observe that the performance of the two-grid method is
always better when applied to the CG and EDG discretizations than when
applied to the HDG discretization. We remark also that the two-grid
method with Lower-Triangular-Element-Wise Vanka relaxation performs
almost as well as the two-grid method with Element-Wise Vanka
relaxation for the CG and EDG discretizations. This is particularly
interesting for larger $k$ since computing the inverse of the
lower-triangular approximation to $K_i$ in \cref{eq:vanka-block} is
significantly cheaper than computing the inverse of $K_i$. Finally, we
remark that always
$(\rho(TG(1,0)))^{\nu_1 + \nu_2} < \rho(TG(\nu_1, \nu_2))$. This
implies that it is cheaper to apply $\nu_1 + \nu_2$ steps of the
two-grid method with one pre-relaxation sweep and no post-relaxation
sweeps than to apply the two-grid method with $\nu_1$ pre-relaxation
sweeps and $\nu_2$ post-relaxation sweeps.

%------------------------------------------------------------------------------
\subsection{Measured multigrid convergence factors}
\label{sec:measured-MG}

\begin{table}[tbp]
  \footnotesize
  \caption{Measured multigrid convergence factors. Notation:
    $\rho_{m, \alpha, \beta}$ is the measured asymptotic convergence
    factor of a two-grid ($\beta=TG$) or $V$-cyle multigrid
    ($\beta=MG$) method applied to a continuous Galerkin ($\alpha=C$),
    HDG ($\alpha=H$) or EDG ($\alpha=E$) discretization. The *
      indicates a case where the measured convergence factor according
      to \cref{eq:def_rhom} with tolerance $10^{-16}$ is not
      representative, since the ratio in \cref{eq:def_rhom} oscillates
      between consecutive iterates. In brackets we denote the value
      computed by $\rho_{m} =(\|{d^{(j)}_h\|}/{\|d^{(0)}_h}\|)^{1/j}$
      which is approximately constant with iteration number
      asymptotically. For each row, the best convergence factor
      among the smoothers that can be executed in parallel (i.e.,
      excluding GS) is indicated in bold.}
    \centering
  \begin{tabular}{l|l|l|l|l|l||l}
    \hline
    $\rho$     &VW  & EW    & JAC    &LTVW   & LTEW   & GS \\
    \hline
    \multicolumn{7} {c} {\footnotesize{$k=1$}} \\
    \hline
    $\rho_{m,C,TG}(\rho_{m,E,TG})$ &0.332    & {\bf 0.197}         & 0.332          &0.332     &0.252             &0.242 \\
    $\rho_{m,C,MG}(\rho_{m,E,MG})$ &0.332    & {\bf 0.196}         & 0.332          &0.332     &0.255             &0.177*(0.261) \\
    \hline
    $\rho_{m,H,TG}$ &{\bf  0.396} &0.461 &0.799 &0.604 &0.698 & 0.382 \\
    $\rho_{m,H,MG}$ & {\bf 0.452} &0.606 &0.800 &0.680 &0.694 & 0.587 \\
    \hline
    \multicolumn{7} {c} {\footnotesize{$k=2$}} \\
    \hline
    $\rho_{m,C,TG}$ &{\bf 0.200} & 0.276 & 0.451 &0.357 &0.212 &0.159\\
    $\rho_{m,C,MG}$ &0.211 & 0.276 & 0.450 &0.357 &{\bf 0.197} &0.162 \\
    \hline
    $\rho_{m,E,TG}$ &0.231 & {\bf 0.192} &0.531 &0.418 &0.316 & 0.241 \\
    $\rho_{m,E,MG}$ &{\bf 0.227} & 0.269 &0.530 &0.418 &0.319 & 0.290 \\
    \hline
    $\rho_{m,H,TG}$ &{\bf  0.433} &0.707 & 0.890 &0.800 &0.794 & 0.616 \\
    $\rho_{m,H,MG}$ & {\bf 0.436} &0.707 & 0.890 &0.800 &0.792 & 0.626 \\
    \hline
    \multicolumn{7} {c} {\footnotesize{$k=3$}} \\
    \hline
    $\rho_{m,C,TG}$ &0.229 &{\bf 0.198}  &0.646 &0.354 &0.296  &0.230 \\
    $\rho_{m,C,MG}$ &0.225 &{\bf0.206}  &0.653 &0.359 &0.314  &0.254 \\
    \hline
    $\rho_{m,E,TG}$ &{\bf 0.281} & 0.324 &0.790 &0.585 &0.570 & 0.446 \\
    $\rho_{m,E,MG}$ &{\bf 0.274} & 0.324 &0.789 &0.608 &0.591 & 0.466 \\
    \hline
    $\rho_{m,H,TG}$ &{\bf 0.472} &0.791 &0.931 &0.860 &0.860 & 0.740 \\
    $\rho_{m,H,MG}$ &{\bf0.472} &0.790 &0.931 &0.859 &0.859 & 0.747 \\
    \hline
  \end{tabular}
  \label{tab:measureasympconv}
\end{table}

\begin{table}[tbp]
  \footnotesize
  \caption{Measured asymptotic convergence factors using different
    grid sizes. Notation: $\rho_{m, \alpha, \beta}$ is the measured
    asymptotic convergence factor of a two-grid ($\beta=TG$) or
    $V$-cyle multigrid ($\beta=MG$) method applied to a continuous
    Galerkin ($\alpha=C$), HDG ($\alpha=H$) or EDG ($\alpha=E$)
    discretization. In all cases we use element-wise Vanka relaxation
    for CG and EDG and vertex-wise relaxation for HDG.}
  \centering
  \begin{tabular}{l||l|l||l|l||l|l}
    \hline
    Mesh     & $\rho_{m,C,TG}$ & $\rho_{m,C,MG}$ & $\rho_{m,E,TG}$ & $\rho_{m,E,MG}$ & $\rho_{m,H,TG}$ & $\rho_{m,H,MG}$ \\
    \hline
    \multicolumn{7} {c} {\footnotesize{$k=1$}} \\
    \hline
    $32^2$   &0.194     &0.195               &0.194         &0.195                  &0.396 &0.429 \\
    $64^2$   &0.197     &0.196               &0.197         &0.196                 &0.396 &0.452 \\
    $128^2$  &0.196     &0.196               &0.196         &0.196                 &0.396 &0.471 \\
    $256^2$  &0.197     &0.197               &0.197         &0.197                &0.395 &0.477 \\
    \hline
    \multicolumn{7} {c} {\footnotesize{$k=2$}} \\
    \hline
    $32^2$ &0.276      &0.277  &0.194 &0.255 &0.432 &0.436 \\
    $64^2$ &0.276      &0.276  &0.192 &0.269 &0.433 &0.436 \\
    $128^2$ &0.277    &0.277   &0.191 &0.270 &0.434 &0.437 \\
    $256^2$ &0.278    &0.278   &0.191 &0.270 &0.434 &0.437 \\

    \hline
    \multicolumn{7} {c} {\footnotesize{$k=3$}} \\
    \hline
    $32^2$ &0.198 &0.207 &0.324 &0.323 &0.471 &0.470 \\
    $64^2$ &0.198 &0.206 &0.324 &0.324 &0.472 &0.472 \\
    $128^2$ &0.199 &0.206 &0.324 &0.324 &0.472 &0.472 \\
    $256^2$ &0.206 &0.206 &0.324 &0.324 &0.472 &0.472 \\
    \hline
  \end{tabular}
  \label{tab:hindep}
\end{table}

In this section we verify the LFA predictions from
\cref{ss:lfaprediction}. For this we consider the Laplace problem on
the unit square domain with homogeneous Dirichlet boundary
conditions. We will solve HDG, EDG, and CG discretizations of this
problem using geometric multigrid starting with a random initial guess
with each component taken from a uniform distribution on $[0,100]$. We
will measure the asymptotic convergence behavior of the proposed
two-grid method as well as its five-level $V$-cyle multigrid
variant. For $\omega$ we use the values from
\cref{tab:Vanka-single-two-grid-LFA}. The finest mesh in our
calculations consists of $64^2 = 4096$ elements and we use an
LU-decomposition to solve the problem on the coarsest grid.

The measured asymptotic convergence factor is given by
\begin{equation}
  \label{eq:def_rhom}
  \rho_m =\frac{\norm[0]{d^{(j)}_h}_2}{\norm[0]{d_h^{(j-1)}}_2},
\end{equation}
where $d_h^{(j)}= f_h - K_h \boldsymbol{\bar{u}}^{(j)}_h$,
$\boldsymbol{\bar{u}}^{(j)}_h$ is the approximation to the solution of
\cref{eq:matsystemred} at the $j^{\text{th}}$ multigrid iteration, and
$j$ is the smallest integer such that
$\norm[0]{d_h^{(j)}}_2 < 10^{-16}$.

\begin{remark}
  \label{measurement-two-options}
  Other approaches to measure the asymptotic convergence factor are to
  define $\rho_m$ as
  $\rho_{m} =(\norm[0]{d^{(j)}_h}/\norm[0]{d^{(0)}_h})^{1/j}$ or
  $\rho_{m} = (\norm[0]{\boldsymbol{\bar{u}}_h^{(j)}} /
  \norm[0]{\boldsymbol{\bar{u}}_h^{(0)}})^{1/j}$. For our problem, the
  results obtained using these measures are almost the same as when
  using \cref{eq:def_rhom}.
\end{remark}

In \cref{tab:measureasympconv} we present the measured asymptotic
convergence factors for the case where we use one pre-relaxation sweep
and no post-relaxation sweeps, i.e., $\nu_1 = 1$ and $\nu_2 = 0$. We
observe that the measured asymptotic convergence factors match very
well with the two-grid LFA convergence factors in
\cref{tab:Vanka-single-two-grid-LFA}, verifying our analysis. In
particular, these results confirm that performance of geometric
multigrid applied to an EDG discretization is similarly effective in
terms of convergence factors and scalability as for geometric
multigrid applied to a continuous Galerkin discretization, and that
this performance is always substantially better than when applied to
an HDG discretization. The results in \cref{tab:measureasympconv} show
that, among the smoothers that can be executed in parallel, the
Vanka-type smoothers may require substantially less multigrid
iterations than the Jacobi smoother, especially as the order of the
discretization increases.

We remark that we also measured the asymptotic convergence factor for
$W$- and $F$-cycle multigrid, but compared to $V$-cycle multigrid
there was only very minor improvement. We furthermore remark that the
conclusions for the $(\nu_1, \nu_2) = (1,0)$ case hold also for other
$(\nu_1, \nu_2)$ combinations. We therefore do not present these
results here.

Finally, in \cref{tab:hindep} we show that the measured asymptotic
convergence factors of the two-grid and $V$-cyle multigrid methods
(with $(\nu_1, \nu_2) = (1,0)$) are independent of the mesh size
$h$. We present the results only for the case where we use
element-wise Vanka relaxation for CG and EDG, and vertex-wise Vanka
relaxation for HDG. We see that the measured two-grid convergence
factors match well with the two-grid LFA predictions from
\cref{tab:Vanka-single-two-grid-LFA}. The measured convergence factors
of $V$-cycle multigrid methods for CG are the same as the two-grid LFA
predictions.  For EDG with $k=2$ and HDG with $k=1$, the measured
$V$-cycle convergence factors are a little worse than the measured
two-grid convergence factors. This is as expected, since in a
$V$-cycle, we use an inexact solver (multigrid) for the coarse
problem, instead of the exact coarse solver of the two-grid method.
All in all, the two-grid and $V$-cycle multigrid methods offer
efficient performance. We remark that for $k=3$, the values
$\rho_{m,\alpha,TG}$ are measured using a three-level $V$-cycle, since
in the two-grid method the $LU$ decomposition for the coarse problem
runs out of memory on the computer we used.

%------------------------------------------------------------------------------
\section{Conclusions}
\label{sec:conclu}

We presented a geometric multigrid method with Jacobi and additive
Vanka relaxation for continuous Galerkin, EDG and HDG discretizations
of the Poisson problem. We used local Fourier analysis to predict the
efficiency of the multigrid method and confirmed what we have observed
in previous work \cite{Rhebergen:2020}, namely that (algebraic)
multigrid performance applied to an EDG discretization outperforms
multigrid applied to an HDG discretization. We therefore conclude that
although EDG does not have the super-convergence properties of an HDG
discretization \cite{Cockburn:2009b}, EDG methods with suitable
iterative methods are competitive alternatives to HDG discretizations.
More generally, our work represents the first analysis of geometric
multigrid convergence for high-order EDG and HDG discretizations in
two dimensions, demonstrating fast convergence, in particular for EDG.

%------------------------------------------------------------------------------
\appendix
\section{Stencil example for HDG with $k=1$}
\label{ap:example_hdg_k1}

We present here an example of the stencils \cref{defi-symbol-operator}
for the HDG method with $k=1$. Using
\cref{eq:block-matrix-to-operator} (or \cref{eq:secd-partition}) and
\cref{defi-symbol-operator} we can write
\begin{equation}
  \label{eq:block-stencil-HDGk1}
  K_h=
  \begin{bmatrix}
[s_{\boldsymbol{\kappa}}]_{X_1X_1} &[s_{\boldsymbol{\kappa}}]_{X_1X_2}     &[s_{\boldsymbol{\kappa}}]_{X_1Y_1} &[s_{\boldsymbol{\kappa}}]_{X_1Y_2} \\
[s_{\boldsymbol{\kappa}}]_{X_2X_1} &[s_{\boldsymbol{\kappa}}]_{X_2X_2}     &[s_{\boldsymbol{\kappa}}]_{X_2Y_1} &[s_{\boldsymbol{\kappa}}]_{X_2Y_2} \\
[s_{\boldsymbol{\kappa}}]_{Y_1X_1} &[s_{\boldsymbol{\kappa}}]_{Y_1X_2}     &[s_{\boldsymbol{\kappa}}]_{Y_1Y_1} &[s_{\boldsymbol{\kappa}}]_{Y_1Y_2} \\
[s_{\boldsymbol{\kappa}}]_{Y_2X_1} &[s_{\boldsymbol{\kappa}}]_{Y_2X_2}     &[s_{\boldsymbol{\kappa}}]_{Y_2Y_1} &[s_{\boldsymbol{\kappa}}]_{Y_2Y_2} \\
 \end{bmatrix}.
\end{equation}
Since the DOFs in HDG are only of type $X$ and $Y$ the different
stencils in \cref{eq:block-stencil-HDGk1} are only of type 1 and 4
(see \cref{ss:partitioningdiscop} for a definition of the four
different stencil types). Using $\alpha = 6k^2$ these stencils are
given by
\begin{align*}
  [s_{\boldsymbol{\kappa}}]_{X_1X_1}
  &=
    \begin{bmatrix}
      -1/24\\
      9/4 \odot\\
      -1/24
    \end{bmatrix},
  &
    [s_{\boldsymbol{\kappa}}]_{X_1X_2}
  &=
    \begin{bmatrix}
      -1/24\\
      11/12\odot\\
      -1/24
    \end{bmatrix},
  \\
  [s_{\boldsymbol{\kappa}}]_{X_1Y_1}
  &=
    \begin{bmatrix}
      -19/24 &       & -7/24\\
      & \odot &   \\
      -7/24  &       &  -1/8
    \end{bmatrix},
  &
    [s_{\boldsymbol{\kappa}}]_{X_1Y_2}
  &=
    \begin{bmatrix}
      -7/24 &         & -1/8\\
      &  \odot  & \\
      -19/24  &       & -7/24
    \end{bmatrix},
\end{align*}
\begin{align*}
  [s_{\boldsymbol{\kappa}}]_{X_2X_1}
  &=
    \begin{bmatrix}
      -1/24\\
      11/12\odot\\
      -1/24
    \end{bmatrix},
  &
    [s_{\boldsymbol{\kappa}}]_{X_2X_2}
  &=
    \begin{bmatrix}
      -1/24\\
      9/4\odot\\
      -1/24
    \end{bmatrix},
\\
  [s_{\boldsymbol{\kappa}}]_{X_2Y_1}
  &=
    \begin{bmatrix}
      -7/24 &          & -19/24\\
      &  \odot   &   \\
      -1/8  &          &  -7/24
    \end{bmatrix},
  &
    [s_{\boldsymbol{\kappa}}]_{X_2Y_2}
  &=
    \begin{bmatrix}
      -1/8 &        & -7/24\\
      & \odot & \\
      -7/24&        & -19/24
    \end{bmatrix},
\end{align*}
\begin{align*}
  [s_{\boldsymbol{\kappa}}]_{Y_1X_1}
  &=
    \begin{bmatrix}
      -1/8 &         & -7/24\\
      & \odot   &   \\
      -7/24  &         &  -19/24
    \end{bmatrix},
  &
    [s_{\boldsymbol{\kappa}}]_{Y_1X_2}
  &=
    \begin{bmatrix}
      -7/24 &          & -1/8\\
      & \odot    & \\
      -19/24 &        & -7/24
    \end{bmatrix},
  \\
  [s_{\boldsymbol{\kappa}}]_{Y_1Y_1}
  &=
    \begin{bmatrix}
      -1/24 & 9/4\odot  & -1/24
    \end{bmatrix},
  &
    [s_{\boldsymbol{\kappa}}]_{Y_1Y_2}
             &=
               \begin{bmatrix}
                 -1/24 & 11/12\odot & -1/24
               \end{bmatrix},
\end{align*}
\begin{align*}
  [s_{\boldsymbol{\kappa}}]_{Y_2X_1}
  &=
    \begin{bmatrix}
      -7/24 &         & -19/24\\
      & \odot    &   \\
      -1/8  &           &  -7/24
    \end{bmatrix},
  &
    [s_{\boldsymbol{\kappa}}]_{Y_2X_2}
  &=
    \begin{bmatrix}
      -19/24 &          & -7/24\\
      & \odot     & \\
      -7/24 &           & -1/8
    \end{bmatrix},
  \\
  [s_{\boldsymbol{\kappa}}]_{Y_2Y_1}
  &=
    \begin{bmatrix}
      -1/24 & 9/4\odot   & -1/24
    \end{bmatrix},
  &
    [s_{\boldsymbol{\kappa}}]_{Y_2Y_2}
             &=
               \begin{bmatrix}
                 -1/24 & 11/12\odot & -1/24
               \end{bmatrix}.
\end{align*}

%------------------------------------------------------------------------------
\bibliographystyle{siamplain}
\bibliography{references}
%------------------------------------------------------------------------------
\end{document}